\documentclass[12pt]{amsart}

\usepackage{fullpage}

\RequirePackage{amsthm,amsmath,amsfonts,amssymb}
\RequirePackage[numbers,sort]{natbib}
          \usepackage{enumerate}
          \usepackage{hyperref}
          \usepackage{color}

\numberwithin{equation}{section}
\theoremstyle{plain}
\newtheorem{Thm}{Theorem}[section]
\newtheorem{Lem}[Thm]{Lemma}
\newtheorem{Prop}[Thm]{Proposition}
\newtheorem{Coro}[Thm]{Corollary}

\theoremstyle{remark}
\newtheorem{Rem}[Thm]{Remark}

\def\cal#1{\mathcal{#1}}
\newcommand{\comment}[1]{}

\newcommand{\ind}{{\bf 1}}

\def\inddd#1{{\ind}_{\left\{#1\right\}}}

\newcommand{\proba}{\mathbb P}
\renewcommand{\P}{\mathbb P}
\newcommand{\esp}{{\mathbb E}}

\newcommand{\inv}{^{-1}}

\newcommand{\eqnh}{\begin{eqnarray*}}
\newcommand{\eqne}{\end{eqnarray*}}
\newcommand{\eqnhn}{\begin{eqnarray}}
\newcommand{\eqnen}{\end{eqnarray}}
\newcommand{\equh}{\begin{equation}}
\newcommand{\eque}{\end{equation}}

\def\summ#1#2#3{\sum_{#1 = #2}^{#3}}

\def\sif#1#2{\sum_{#1=#2}^\infty}

\newcommand{\eqd}{\stackrel{d}{=}}

\def\topp#1{^{(#1)}}

\def\nn#1{{\left\|#1\right\|}}
\def\snn#1{\|#1\|}

\def\abs#1{\left|#1\right|}

\def\ccbb#1{\left\{#1\right\}}

\def\pp#1{\left(#1\right)}
\def\spp#1{(#1)}

\def\mmid{\;\middle\vert\;}

\def\floor#1{\left\lfloor #1 \right\rfloor}
\def\sfloor#1{\lfloor #1 \rfloor}
\def\ceil#1{\left\lceil #1 \right\rceil}

\def\vv#1{{\boldsymbol #1}}

\def\vvx{{\vv x}}
\newcommand{\vvX}{{\vv X}}

\def\qmand{\quad\mbox{ and }\quad}

\def\mwith{\mbox{ with }}
\def\qmwith{\quad\mbox{ with }\quad}
\def\mfa{\mbox{ for all }}

\def\mmas{\mbox{ as }}

\def\wt#1{\widetilde{#1}}
\def\wb#1{\overline{#1}}
\def\what#1{\widehat{#1}}

\def\limn{\lim_{n\to\infty}}
\def\limm{\lim_{m\to\infty}}

\def\limsupn{\limsup_{n\to\infty}}

\def\weakto{\Rightarrow}

\def\Z{{\mathbb Z}}

\def\R{{\mathbb R}}

\def\N{{\mathbb N}}

\def\calA{\mathcal A}
\def\calB{\mathcal B}
\def\calC{\mathcal C}

\def\calF{\mathcal F}

\def\calI{\mathcal I}

\def\calL{\mathcal L}
\def\calM{\mathcal M}

\def\calV{\mathcal V}
\def\calW{\mathcal W}

\def\cal#1{\mathcal{#1}}
\def\topp#1{^{\scriptscriptstyle (#1)}}

\def\Leb{{\rm Leb}}

\def\ddelta#1{\delta_{\pp{#1}}}
\def\PPP{{\rm PPP}}

\newcommand{\Clipb}{C_{{\rm Lip},{\rm b}}}

\newcommand{\EqD}{\overset{d}{=}}

\newcommand{\AC}{{\mathcal{AC}}}

\begin{document}\sloppy

\title[A remarkable example on clustering of extremes]{
A remarkable example on clustering of extremes for regularly-varying stochastic processes}

\author{Shuyang Bai}\address{Shuyang Bai\\
Department of Statistics \\ University of Georgia\\
310 Herty Drive\\
Athens, GA, 30602, USA.}
\email{bsy9142@uga.edu}
\author{Rafa{\l} Kulik}\address{Rafa\l~Kulik\\
Department of Mathematics and Statistics\\
University of Ottawa\\
STEM Complex\\
150 Louis Pasteur Private\\
K1N 6N5 Ottawa, Ontario, Canada.}
\email{rkulik@uottawa.ca}

\author{Yizao Wang}\address{Yizao Wang\\Department of Mathematical Sciences\\University of Cincinnati\\2815 Commons Way\\Cincinnati, OH, 45221-0025, USA.}\email{yizao.wang@uc.edu}

\date{\today}
\begin{abstract}
The stable-regenerative multiple-stable model has been shown recently to have distinct candidate extremal index and extremal index. To understand further this rare phenomenon, two more results are established here for the double-stable model. The first is the convergence of point processes for the clusters of extremes, enhancing the previous result on the weak convergence of random sup-measures. Most interestingly, the second result reveals a new phase transition at the mesoscopic level when computing the asymptotic exceedance probability over a block, $\proba(\max_{k=1,\dots,d_n} X_k>b_n)$, as $n\to\infty$. Here, the mesoscopic level is referred to the fact that the block size $d_n$ is allowed to grow at the rate $n^\rho$ with $\rho\in[0,1]$, while the  threshold $b_n$ is such that $\proba(X_1>b_n)\sim 1/n$. The recently discovered discrepancy between the candidate extremal index and the extremal index is shown to be just a reflection of this phase transition that is prohibited by the anticlustering condition.
\end{abstract}
\maketitle
%
%
%
%

\section{Introduction and main results}
\subsection{Extremes of regularly varying stochastic processes} 
Let $\pp{X_k}_{k\in\N_0}$ be a stationary sequence of  non-negative random variables whose finite-dimensional distributions are jointly regularly varying with a tail index $\alpha>0$. The assumption, in particular, implies a power-law tail of the marginal distribution $\proba(X_0>x)\sim x^{-\alpha} L(x)$ as $x\to\infty$, where $L$ is a  function  slowly varying at infinity. Such processes are also known as regularly varying stochastic processes. We refer the reader to recent monographs \cite{kulik20heavy,mikosch24extreme} for a comprehensive treatment of these processes.
How do their extreme values behave? This question has a long history since the 1980s. The tail index $\alpha$ matters and, of course, the dependence of the sequence also matters. Here we focus on the regime in which the sequence has {\em weak extremal dependence}, 
which will be explained  
after \eqref{eq:macro} below
in a moment. For the sake of simplicity, we consider nonnegative processes in this paper. This choice does not change the qualitative behavior that we shall reveal, but facilitates the presentation by simplifying the analysis at a few places.

  Generally, there are two types of results one aims at establishing. One is to characterize the global extremes by proving macroscopic limit theorems. In this direction, a first result is of the type
  \equh\label{eq:macro}
\limn  \proba\pp{\frac1{b_n}\max_{k=1,\dots,n}X_k\le x} = \exp\pp{-\theta
 x^{-\alpha}},
  \eque
  with $b_n$ be such that $\lim_{n\to\infty}n\proba(X_0>b_n)=1$, and certain constant $\theta\in(0,1]$. When \eqref{eq:macro} holds, we say the sequence has weak extremal dependence  (for the sake of simplicity we shall simply say {\em weak dependence} in the sequel).
   A more refined limit theorem in place of \eqref{eq:macro} is the convergence of point process of $\summ k1n \ddelta{X_k/b_n,k/n}$ in the space of Radon point measures on $(0,\infty)\times[0,1]$. The constant $\theta$ in \eqref{eq:macro} is referred to as the {\em extremal index}, which should be viewed as a measure of weak dependence at macroscopic level: we have $\theta=1$  typically when the stationary sequence consists of extremally independent random variables.

  The other type of results concerns microscopic limit theorems characterizing local clustering of extremes. 
  One such result takes the form of
  \equh\label{eq:micro}
\calL\pp{\left(\frac{X_k}{ X_0 }\right)_{k=0,\ldots,m} \mmid  X_0 >x}  \Rightarrow
\calL\pp{ \left(\Theta_k\right)_{k=0,\ldots,m} },
\eque
for all $m\in\N$, where $\{\Theta_k\}_{k\in\N_0}$ is known as the spectral tail process of $\pp{X_k}_{n\in\N_0}$ in the literature (e.g.~\citep[Proposition 5.2.1]{kulik20heavy}). From the non-negative spectral tail process, one can compute the {\em candidate extremal index} of the original process ({\cite[Proposition 5.6.5]{kulik20heavy}}):
\[
\vartheta:=\esp\pp{\sup_{k\ge 0}\Theta_k^\alpha - \sup_{k\ge 1}\Theta_k^\alpha}.
\]
The spectral tail process provides a probabilistic characterization of the local clustering of extremes. Based on this process, the candidate extremal index should be viewed as a measure of local dependence. 

Extremes for stationary processes with weak dependence have been extensively investigated since the 1980s. A classical monograph was by \citet{leadbetter83extremes}, and the notion of extremal index has been extensively investigated there. The notion of tail process was later introduced in the seminal work by \citet{basrak09regularly}. The extremal index $\theta$ and candidate extremal index $\vartheta$ are two different numerics characterizing global (macroscopic) and local (microscopic) behaviors of extremes, respectively, and hence they do not necessarily agree, although the relation $\theta\le \vartheta$ always holds  \citep[Lemma 7.5.4]{kulik20heavy}. 
In fact,
for most examples of regularly varying processes with weak dependence investigated so far in the literature, the two indices are the same; this fact is well known by now to follow from a general framework by verifying that the processes satisfy a mixing-type condition and a so-called anticlustering condition. Moreover, this has become the standard approach in the literature to show \eqref{eq:macro}: 
first,
 one characterizes local clustering of extremes by establishing the convergence of (spectral) tail process in
 \eqref{eq:micro}, then one establishes the mixing-type condition and the anticlustering condition; see for example \citep{kulik20heavy,mikosch24extreme}. 

A recent key development is based on a remarkable process, known as the {\em stable-regenerative multiple-stable model},  studied in \citep{bai23tail,bai24phase}.  
This example stems from the studies of stochastic processes with long-range dependence \citep{samorodnitsky16stochastic,pipiras17long}. 
The model in the so-called subcritical regime has weak dependence in the sense of \eqref{eq:macro}; however, it is qualitatively different from all the other examples we are aware of in the sense that 
\[
\theta<\vartheta.
\]
This model provides challenges and poses new questions. First, the aforementioned standard approach to establish \eqref{eq:macro} no longer works. Second, it is natural to ask for a further explanation on the discrepancy between the two indices. 

We continue the recent investigations in \citep{bai24phase,bai23tail}. 
The contributions of this paper are the following on the stable-regenerative {\em double-stable} model (multiplicity $p=2$, and also in the subcritical regime in the sense of \citep{bai24phase} which we do not repeat in the sequel). 
In the sequel, we write $a_n\sim b_n$ if $\limn a_n/b_n = 1$. 
\begin{enumerate}[(i)]
\item We provide a point-process convergence of cluster of extremes for the model, a result strictly stronger than \eqref{eq:macro}, which was established in \citep{bai24phase}. Our proof is based on several steps of approximations, including a 
de-aggregation step and a Poisson approximation by the two-moment method. The same 
ideas have
 been applied in \citep{bai24phase} for the first time, and several improvements have been made here.
\item 
We reveal 
a new and
 more refined phase transition for the extremes, at the {\em mesoscopic level}. We introduce a statistic $\vartheta_\rho$ on the extremes that interpolates between 
\[
\vartheta = \vartheta_0 \qmand \theta= \vartheta_1.
\] More precisely,  we establish {for $d_n\to\infty, d_n = o(n)$, and $\limn\log d_n/\log n = \rho\in[0,1]$,}
\equh\label{eq:meso}
\proba\pp{\frac1{b_n}\max_{k=1,\dots,d_n}X_j>y} \sim \vartheta_\rho d_n\proba\pp{\frac{X_0}{b_n}> y}.
\eque
Such a regime is not 
present under a general anticlustering condition that is satisfied by most of the models investigated in the literature (under such a condition, \eqref{eq:meso} holds with $\vartheta_\rho$ replaced by $\theta$ for all $\rho\in[0,1]$). We will call $\vartheta_\rho$ a {\em mesoscopic extremal index}.
\item Our model is essentially an aggregation model. At the technical level, we reveal that extremes of which parts/layers of the model contribute to the mesoscopic limit (of probability of order $d_n/n$) but not to the macroscopic limit (of probability $O(1)$). 
\end{enumerate}
In words, this paper and \citep{bai24phase,bai23tail} all combined provide a non-trivial example that shows that when the anticlustering condition is violated, a new regime regarding the extremes behavior may appear between microscopic and macroscopic levels. It is left to further research that whether other models with weak extremal dependence but violating the anticlustering condition may exhibit a similar phase transition at mesoscopic level as revealed here. 
\begin{Rem}
Stable-regenerative models have attracted interest for their own sake. The original stable-regenerative model (with multiplicity one) was introduced by \citet{rosinski96classes}. It exhibits long-range dependence and has become a classic example in the studies of stochastic processes with long-range dependence. The extremes for the original model were first extensively studied in \citep{samorodnitsky04extreme,samorodnitsky19extremal}. The multiple-stable version of the model was then introduced in \citep{bai20functional}, and the model may exhibit either short-range (weak) or long-range dependence depending on the relation between the multiplicity parameter $p\in\N$ and the renewal parameter $\beta\in(0,1)$, as shown in \citep{bai23tail,bai24phase}. In particular, the subcritical regime $p\beta-p+1<0$ has been shown to exhibit a weak dependence in the sense of \eqref{eq:macro}. Here, we study the double-stable model (so $p=2$) in the subcritical regime (so $\beta<1/2$). Another interesting direction is to consider extremes of stable-regenerative models with other types of tails, and the extremes have recently been shown to have new scaling limits \citep{chen22extremal,chen24new}.
\end{Rem}

\subsection{The stable-regenerative double-stable model}  

The family of processes of our interest has a tail parameter $\alpha\in(0,1)$, a memory parameter $\beta\in(0,1/2)$.  The representation is intrinsically related to renewal processes, 
 for which we introduce some notation. Consider a discrete-time renewal process  with the consecutive renewal times denoted by 
 $\vv\tau:=\{\tau_0,\tau_1,\dots\} \subset \N_0:=\{0,1,2,\ldots\}$.
  Here $\tau_0$ is the initial renewal time, and the inter-renewal times $(\tau_{i}-\tau_{i-1})_{i\ge 1}$ are i.i.d.\  $\N$-valued with cumulative distribution function $F$, that is, $F(x) = \proba(\tau_{i}-\tau_{i-1}\le x), i\in\N, x\ge 0$. Denote the probability mass function by $f(n) = \proba(\tau_{i}-\tau_{i-1} = n), n\in\N$. Throughout, we assume
\equh\label{eq:F}
\wb F(x) = 1-F(x) \sim \mathsf C_Fx^{-\beta} \mmas x\to\infty \qmwith \beta\in(0,1/2),
\eque
which  implies  an infinite mean,
and the following technical assumption
\[
\sup_{n\in\N}\frac{nf(n)}{\wb F(n)}<\infty.
\]
 By default, a renewal process starts at renewal at time 0, and hence $\tau_0 = 0$. 
 Note that our renewal processes may be {\em delayed}, that is, $\tau_0$  is not necessarily zero, and may be a random variable in $\N_0 = \N\cup\{0\}$, and we shall be specific when this is the case.
An important notion is the {\em stationary delay measure} of the renewal process, denoted by $\pi$. More precisely, $\pi$ is supported on $\N_0$ with
\[
\pi(k)\equiv \pi(\{k\}) = \wb F(k) = 1-F(k), \quad k\in\N_0.
\]
(For the sake of simplicity, we do not distinguish $\pi(k)$, the mass function at $k\in\N_0$, from $\pi(\{k\})$, the measure evaluated at the set $\{k\}$.) Note that the stationary delay  measure $\pi$ is a $\sigma$-finite and infinite measure on 
$\N_0$,
 since the renewal distribution has infinite mean.

We then provide a series representation of the model.  
Consider
\[
\sif i1\ddelta{y_i,d_i}\eqd \PPP\pp{(0,\infty]\times\N, \alpha x^{-\alpha-1}dxd\pi},
\]
where the right-hand side is understood as a Poisson point process on $(0,\infty]\times \N_0$ with intensity $\alpha x^{-\alpha-1}dxd\pi$. 
In addition, let $\{\vv\tau\topp{i,d_i}\}_{i\in\N}$ denote, given the above Poisson point process, conditionally independent delayed renewal processes, each $\vv\tau\topp{i,d_i}$ with initial renewal time delayed at
$\tau_0 = d_i$ and   inter-renewal times following $F$. Write also $\vv \tau^{(i)}=\vv\tau^{(i,0)}$, $i\in \N$, namely, i.i.d.\ copies of $\vv \tau$ which starts at the origin.

In this paper,  we consider the {\em stable-regenerative double-stable model} given by
\equh\label{eq:series infty p}
\pp{X_k}_{k\in\N_0} \eqd\pp{\sum_{0<i_1<i_2<\infty} y_{i_1}y_{i_2} \inddd{k\in \vv\tau\topp{i_1,d_{i_1}}\cap \vv\tau\topp{i_2,d_{i_2}}   }}_{k\in\N_0}.
\eque
The time index $k$ of the  stationary process defined above can be extended to $k\in \Z$ by Kolmogorov extension theorem; in fact, an explicit representation using two-sided renewals can be written down as in  \cite[Section 2.3]{bai23tail}. 
\begin{Rem}\label{rem:nonneg}
Note that this model differs from the one considered in \cite{bai23tail} in that the Poisson jumps $y_i$'s there were instead governed by a symmetric mean $|x|^{-\alpha-1}dx$ on $\R\setminus\{0\}$  with a wider range of $\alpha\in (0,2)$, so that $y_i$'s also take negative values. We consider this modified model \eqref{eq:series infty p}  to avoid the distraction of handling negative values, and the phenomena of interest are mainly due to the dependence structure  induced by $\vv\tau\topp{i_1,d_{i_1}}\cap \vv\tau\topp{i_2,d_{i_2}}$, $i_1<i_2$.
\end{Rem}
It is known (\cite{doney97onesided}; see also \cite{bai23tail}) that (recall $\mathsf C_F$ from \eqref{eq:F}) 
\begin{equation}\label{eq:u(n) asymp}
u(n) := \proba(n\in\vv\tau\topp{1,0})\sim \frac{n^{\beta-1}}{\mathsf C_F \Gamma(\beta)\Gamma(1-\beta)},
\end{equation}
as $n\rightarrow\infty$, and with $\beta\in(0,1/2$), we have
\equh\label{eq:qFp}
\mathsf q_{F,2} := 
\left(\sif n0 u(n)^2\right)^{-1}  \in(0,1),
\eque 
which is the parameter of the geometric random variable $|\vv \tau^{(1)}\cap \vv \tau^{(2)}|$, i.e., 
{$\proba(|\vv\tau\topp1\cap\vv\tau\topp2| = k) = (1-\mathsf q_{F,2})^{k-1}\mathsf q_{F,2}, k\in\N$,} and $\esp |\vv \tau^{(1)}\cap \vv \tau^{(2)}|=1/\mathsf q_{F,2}$; see \cite[Section 2.1]{bai23tail}.

As mentioned in Remark \ref{rem:nonneg}, the process $\pp{X_k}_{k\in\N_0}$  differs from  the one considered in \cite{bai23tail} in that the latter can take negative values. Nevertheless,  many results in \cite{bai23tail} can be extended to  the current model.  
It follows from an argument similar to \cite[Lemma 3.1]{bai23tail} that as $x\rightarrow\infty$,
$\proba\pp{X_0>x}\sim \alpha  x^{-\alpha} \log(x)/2$. 
Hence with
\equh\label{eq:b_n}
b_n = \pp{\frac12n\log  n}^{1/\alpha},
\eque
we have
\begin{equation}\label{eq:X_0 tail}
\proba(X_0>b_n x)\sim \frac{x^{-\alpha}}{n} , \quad x>0,
\end{equation}
as $n\rightarrow\infty$.  Note that here  $b_n$ differs from the one  in \citep{bai23tail} by a multiplicative factor of $2$ inside the parenthesis in \eqref{eq:b_n}. 

Set 
\begin{equation}\label{eq:Theta_k}
{\vv \Theta := \pp{\Theta_k}_{k\in\N_0}} \qmwith \Theta_k:=\inddd{k\in \vv\tau\topp 1 \cap \vv\tau\topp 2  }, \quad k\in \N_0.
\end{equation}
Note that $\Theta_0 = 1$ by definition.  
 The following result is an adaptation of \cite[Theorem 1]{bai23tail} whose   proof is a straightforward modification that we shall omit. 
\begin{Thm}\label{Thm:tail proc}
For any $m\in \N_0$, the convergence \eqref{eq:micro} holds with the {one-sided} tail process $\vv \Theta$ in  \eqref{eq:Theta_k}. In particular, the candidate extremal index is 
 \[\vartheta=\mathsf q_{F,2}.\]
\end{Thm}

\subsection{Main results}\label{sec:main}
Our first result concerns the macroscopic limit of extremes. Set $\ell_0:=\{\vv x\equiv\pp{x_k}_{k\in\Z}:\lim _{k\to\pm\infty}x_k = 0\}$. Let $\wt{\ell}_0$ denote the space of equivalent classes of $\ell_0$ with the equivalence relation $\sim$   given by $\vv {x}\sim \vv y$ if and only if 
\[
 \mathsf B^j \vv x=\vv y
\]
  for some $j\in\Z$, where $\mathsf B$ is the backshift operator defined by $(\mathsf {B}\vv x)_{n}=   {x}_{n-1}$, $n\in \Z$.  The space $\wt\ell_0$ can be made into a complete and separable metric space. We write $[\vv x] :=\{\vv y\in\ell_0:\vv x\sim\vv y\}$ and set 
\equh\label{eq:S}
S := \wt\ell_0\setminus\ccbb{[\vv 0]}. 
\eque
For more details, see Appendix \ref{sec:topo}.

  Set 
  \[
  \calI_{d_n,j}:=\{(j-1)d_n+1,\dots,jd_n\}, j=1,\dots,k_n,
  \]
  with $d_n\rightarrow\infty$, $d_n=o(n)$ as $n\rightarrow\infty$, and 
  throughout we write
\[
k_n:=\floor{\frac n{d_n}}.
\]
 Set 
 \[
  \vv X_{d_n,j}: = \pp{X_{(j-1)d_n+1},X_{(j-1)d_n+2},\dots,X_{jd_n}}, j=1,\dots,k_n.
\]

In the statement of our main results below, we shall need the so-called {\em conditional spectral tail process} \citep[Definition 5.4.6]{kulik20heavy} associated to the process $\{X_k\}_{k\in\Z}$ (extended to a two-sided one). Following the standard notation in the literature, we let $\vv Q$ denote the conditional spectral tail process, which is a random element in $S$. In particular, the law of $\vv Q$ is defined by a change of measure to the law of {\em the two-sided spectral tail process}. Interestingly, the law of $\vv Q$ obtained this way turned out to be the law of the one-sided spectral tail process $\vv\Theta$ viewed as an element in $S$, as shown in Lemma \ref{lem:Q = Theta}. 
Therefore, we shall write
\equh\label{eq:Q = Theta}
\vv Q = \vv\Theta = \pp{\Theta_k}_{k\in\N_0}
\eque
in the sequel with a little abuse of language: $\vv\Theta$ is a one-sided stochastic process indexed by $k\in\N_0$, while $\vv Q$ is its corresponding element in $S$. 
We do not need any properties of or results on conditional spectral tail process in our proofs, but only the explicit description of the law of $\vv Q$ in \eqref{eq:Q = Theta} (i.e., the law of $(\Theta_k)_{k\in\N_0}$ in \eqref{eq:Theta_k}, which is needed in Lemma \ref{Lem:conv Q}). We mention the notion of conditional spectral tail process so that our results can be compared directly to other results involving $\vv Q$ in the literature, see for example \citep[Chapters 6 and 7]{kulik20heavy}.


In the following, suppose $\vv{Q}_\ell$ are i.i.d.\ copies of the {conditional} spectral tail process $\vv Q$.
We shall view each $\vv X_{d_n,j}$ and $\vv{Q}_\ell$ as an element in $\wt\ell_0$ representing its equivalence class (strictly speaking, first as an element in $\ell_0$ by padding zeros on both and left sides of the vector respectively).

\begin{Thm}\label{thm:1}
The following weak convergence of  point processes   on $  S\times [0,1]$ holds for any sequence of integers $\{d_n\}_{n\in\N}$ such that $d_n\to\infty, n/d_n\to\infty$, as $n\rightarrow\infty$:
\equh\label{eq:PPC}
\xi_n:=\summ j1{k_n} \ddelta{\vv  X_{d_n,j}  /b_n, \ j/k_n    }
\weakto  \xi:=\sif\ell1 \ddelta{  \vv{Q}_\ell (\Gamma_\ell/\theta)^{-1/\alpha},\ U_\ell     },
\eque
  with $S$ is as in \eqref{eq:S}, where $\theta=  \vartheta_1 =
  (1-2\beta) \mathsf{q}
    _{F,2}$, $\{U_\ell\}_{\ell\in\N}$ are i.i.d.~random variables uniformly distributed over $(0,1)$, $\{\Gamma_\ell\}_{\ell\in\N}$ are consecutive arrival times of a standard Poisson process, and the two families and $\{\vv Q_\ell\}_{\ell\in\N}$ are independent.
\end{Thm}

Our second result concerns the mesoscopic limit of extremes. We present a weaker version here for the sake of simplicity. The full version in  Theorem \ref{thm:2} concerns the convergence of clusters (which implies Theorem \ref{thm:2'} below immediately {for $\rho\in(0,1]$}; see the proof of \cite[Corollary 6.2.6]{kulik20heavy}), and requires certain topological background that might be a distraction for the discussions here.
With $\mathsf{q}_{F,2}$ defined in \eqref{eq:qFp}, set
\[
 \vartheta_\rho:=  (1-2\rho \beta)\mathsf{q}_{F,2}, \quad\rho\in[0,1].
\]
Notice that $\vartheta_0 = \vartheta$ and $\vartheta_1 = \theta$.
\begin{Thm}\label{thm:2'}
Let $\{b_n\}_{n\in\N}$ be as in \eqref{eq:b_n}.
Under the assumption 
\equh\label{eq:d_n}
d_n \to\infty,\quad \limn\frac{\log d_n}{\log n} = \rho\in[0,1], \qmand d_n = o(n) \mbox{ if } \rho = 1,
\eque
for all $x>0$,
\equh\label{eq:block EVT}
\proba\pp{\frac1{b_n}\max_{k=1,\dots,d_n}X_k> x  }\sim\frac{  \vartheta_\rho x^{-\alpha}}{k_n} \sim \vartheta_\rho d_n\proba\pp{\frac{X_0}{b_n}>x},
\eque
as $n\to\infty$.
\end{Thm}
{Note that \eqref{eq:d_n} is slightly weaker than $d_n = n^\rho L(n)$ for a slowly varying function $L$ at infinity.}

The mesoscopic limit theorem actually interpolates microscopic and macroscopic limit theorems as follows.
A counterpart of Theorem \ref{thm:2'} with fixed $d$ can be easily derived using convergence of tail processes, as explained in Lemma \ref{lem:d fixed} (see also the comment afterwards how this is related to $\rho=0$). With $\rho=1$, Theorem \ref{thm:2'} can be related to the macroscopic limit \eqref{eq:macro} in the sense that in the presence of asymptotic tail independence of block maxima, which we believe to hold when $\log d_n/\log n\to 1$ and $d_n = o(n)$, one has
\[
\proba\pp{\frac1{b_n}\max_{k=1,\dots,n}X_k\le x} \sim \pp{1-\proba\pp{\frac1{b_n}\max_{k=1,\dots,d_n}X_k>x}}^{\floor{n/d_n}}\to \exp\pp{-\vartheta_1 x^{-\alpha}}.
\]

\subsection{Comments}
We conclude the introduction with a few comments on our main results. The first two concern the so-called anticlustering condition. 
Recall that given two sequences $\{b_n\}_{n\in\N}$ and $\{d_n\}_{n\in\N}$, we say the anticlustering condition $\AC(d_n,b_n)$ holds, if  
\[
\limm \limsupn \proba\pp{\max_{m  \le |j|\le d_n}X_j>b_n x\mmid X_0>b_n y} = 0, \mfa x,y>0.
\]

First, Theorem \ref{thm:2'} yields that for the anticlustering condition $\calA\calC(d_n,b_n)$ to hold, necessarily $d_n = o(n^\epsilon)$ for all $\epsilon>0$. See Corollary \ref{Cor:ac fails}. A sufficient condition $d_n=o(\log n)$ can be established by adapting the proof of \citep[Lemma 5.2]{bai23tail}. In contrast, for i.i.d.~sequence, as well as for most examples of weakly dependent sequences, the anticlustering condition holds as long as $d_n = o(n)$. This is an extensively investigated and used criteria in the literature for regularly varying stochastic processes with weak dependence. For a continuous version, see \citet{soulier22tail}. For an extension to regularly
varying random fields, see \citet{planinic23palm}.

\begin{Coro}\label{Cor:ac fails}
Let $\{b_n\}_{n\in\N}$ be as in \eqref{eq:b_n} and suppose that the block length sequence   $\{d_n\}_{n\in\N}$  satisfies $d_n\gg n^{\epsilon}$ for some $\epsilon>0$ as $n\rightarrow\infty$. Then
the anticlustering condition $\mathcal{AC}(d_n,b_n)$ fails. 
\end{Coro}
\begin{proof}
Suppose  $\mathcal{AC}(d_n,b_n)$ holds with $d_n\sim C n^{\epsilon}$ some $\epsilon\in (0,1)$.  {By \cite[Corollary 6.2.6]{kulik20heavy}}, the anticlustering condition gives
\[
\frac{\proba(\max_{k=1,\ldots,d_n}X_k>b_n x)}{d_n\proba(X_0>b_n)}\sim \vartheta_0 x^{-\alpha}. 
\]
At the same time, \eqref{eq:X_0 tail} gives
\[
d_n\proba(X_0>b_n)\sim \frac{d_n}n \sim k_n^{-1}.
\] 
Thus,
\equh\label{eq:classic}
\proba\left(\frac1{b_n}\max_{k=1,\ldots,d_n}X_k> x\right)\sim \vartheta_0 \frac{x^{-\alpha}}{k_n}, 
\eque
contradicting \eqref{eq:block EVT} since $\vartheta_0>\vartheta_{\epsilon}$. Thus, the series does not satisfy the anticlustering condition with $d_n\sim C n^{\epsilon}$ and hence with any block size larger than $n^{\epsilon}$.
\end{proof}

Second, notice that while the rate of the block size $d_n$ has an impact at the mesoscopic level, {it has no effect on the macroscopic limit.}
{So the index $\rho$ only concerns the mesoscopic level.}
 This might come as a confusing remark at the first sight. In words, this phenomenon is due to the co-existence of two features of the model: the power-law decay of the tails and the  aggregation nature, and can be viewed as another phase transition {\em at mesoscopic level within the regime of weak dependence} that is not present when the anticlustering condition $\calA\calC(d_n,b_n)$ holds with $d_n = o(n)$ (in which case, \eqref{eq:classic} holds instead of \eqref{eq:block EVT}). 
  This is better clarified by 
 comparing the proofs of Theorems \ref{thm:1} and \ref{thm:2'} in Section \ref{sec:overview 2}; {see Remark \ref{rem:why}}. It would be very interesting to know whether such a phase transition exists for any other regularly varying stochastic processes. It would also be interesting to know whether there are models such that the anticlustering condition holds with {$\limn\log d_n/\log n = \rho$} for $\rho$ up to certain $\rho_0<1$ only. 

Third, Theorem \ref{thm:1} implies the commonly studied point-process convergence stated in Corollary \ref{coro:1'}, which implies certain functional central limit theorems and extremal limit theorems immediately. 

\begin{Coro}\label{coro:1'}
Following the notation in Theorem \ref{thm:1}, the following weak convergence of  point processes   on $ (0,\infty) \times [0,1]$ holds as $n\rightarrow\infty$:
\[
 \summ j1{n} \ddelta{  X_{j}  /b_n, \ j/n    } \weakto   \sif\ell1  \ddelta{  G_{\beta,\ell} (\Gamma_\ell/\theta)^{-1/\alpha},\ U_\ell     },
\]
where $\{G_{\beta,\ell}\}_{\ell\in\N}$ are i.i.d.~geometric random variables with parameter $\mathsf q_{F,2}$. 
\end{Coro}

Combining Corollary \ref{coro:1'} above, \cite[Theorem 8.3.1]{kulik20heavy} (point-process convergence implies finite-dimensional convergence of partial sums when $\alpha\in (0,1)$) and \cite[Corollary C.3.8]{kulik20heavy} (convergence of finite-dimensional distributions of nondecreasing processes yields functional convergence in $M_1$-topology), we obtain the following functional central limit theorem.
\begin{Coro}\label{Coro:CLT}
As $n\rightarrow\infty$,  
\[
\pp{\frac{1}{b_n}\summ k1{\floor{nt}} X_k }_{t\in [0,1]}\weakto  \pp{\theta^{1/\alpha} \sif\ell1 G_{\beta,\ell} \Gamma_{\ell}^{-1/\alpha}\inddd{U_\ell\le t} }_{t\in [0,1]},
\]
where $\weakto$ denotes the weak convergence in $\mathbb{D}([0,1])$ space equipped with $M_1$-topology. 
The limiting process on the right-hand side is an $\alpha$-stable subordinator. 
\end{Coro}
 Corollary  \ref{Coro:CLT} above can be viewed as a complementary result of \citep{bai20functional}, where for the same model (strictly speaking, the models are also different as explained in Remark \ref{rem:nonneg}) but with $\beta \in (1/2,1)$ a functional central limit theorem is established. With $\beta\in(1/2,1)$, this is called the {\em supercritical regime}, and in this regime the limit is not an $\alpha$-stable process.

One can also establish the convergence of random sup-measures \citep{vervaat88random}, and recover the results in \citep{bai24phase}. The convergence of random sup-measures implies in particular \eqref{eq:macro}. 
The fact that Theorem \ref{thm:1} implies the convergence of random sup-measures is straightforward, 
and we omit the details. {A few additional consequences of the main results that are of more technical nature are collected in Section \ref{sec:other}}. 

{\em The rest of the paper is organized as follows.} In Section \ref{sec:PP} we prove Theorem \ref{thm:1}. In Section \ref{sec:single} we state and prove the full version of Theorem \ref{thm:2'} in Theorem \ref{thm:2}. We also comment on the relations of the two proofs in Section \ref{sec:overview 2}, and in particular the discussion therein explains which part of the model contributes to the mesoscopic limit but not the macroscopic limit.

We will use the following notation throughout this paper. We use $C$ to denote a generic positive constant  whose value may change from line to line, we use both $a_n=o(b_n)$ and $a_n\ll b_n$ to denote $\lim_{n\rightarrow\infty} a_n/b_n=0$ for two positive sequences $\{a_n\}_{n\in\N}$ and $\{b_n\}_{n\in\N}$, $n\in \N:=\{1,2,\ldots\}$,

\subsection*{Ackowledgements}
 Y.W.'s research was partially supported by
Army Research Office, USA (W911NF-20-1-0139),  Simons Foundation (MP-TSM-00002359), 
and the Taft Center Fellowship (2024--2025) from Taft Research Center at University of Cincinnati.


\section{Point-process convergence of clusters of extremes}\label{sec:PP}

 In this section, we prove Theorem \ref{thm:1}. The proof is based on a triangular-array representation of the model, which is provided in Section \ref{sec:renewal}. The same representation will be needed in Section \ref{sec:single} as well. 

\subsection{A triangular-array representation}\label{sec:renewal}
We introduce a different representation of $ \pp{X_k}_{k\in\N}$ in \eqref{eq:series infty p} that shall be needed for our proofs, for finite-dimensional distributions of the process.
We are interested in the joint law of $(X_1,\dots,X_n)$, for some $n\in \N$. Then, for all those $i\in\N$ such that
$\vv\tau\topp{i,d_i}\cap\{1,\dots,n\} = \emptyset$, they do not contribute to the series representation. (Recall also that $d_i$ is a strictly positive integer throughout.)
Therefore, by a standard thinning argument of Poisson point processes, it follows that
\equh\label{eq:thinning}
\sif i1\ddelta{  \Gamma_i^{-1/\alpha},\ \vv\tau\topp{i,d_i}\cap\{1,\dots,n\}}\inddd{\vv\tau\topp{i,d_i}\cap\{1,\dots,n\}\ne\emptyset}\eqd\sif i1\ddelta{w_n^{1/\alpha} \Gamma_i^{-1/\alpha},\ R_{n,i}},
\eque
where on the right-hand side, 
\[
w_n = \summ k0{n-1} \pi(k) =  \summ k0{n-1} \wb F(k) \sim \frac{\mathsf C_F}{1-\beta}\cdot n^{1-\beta},
\]
as $n\to\infty$,
the random variables $\{\Gamma_i\}_{i\in\N}$ are consecutive arrival times of a standard Poisson process,  $\{R_{n,i}\}_{i\in\N}$ are i.i.d.~random closed subsets of $\{1,\dots,n\}$ with the law $R_{n,i}\eqd R_n$ described below, and all families are independent. 

Suppose $\vv\tau^*$ is a delayed renewal process with stationary delay $\pi$ and renewal distribution $F$. This is not a classical random element, as the total is infinite (since the stationary delays is so). Nevertheless, it induces an infinite measure on the space $\calF_0(\N)$ of nonempty closed subsets of $\N$ with Fell topology \citep{molchanov17theory}, and we let $\wt\mu^*$ denote this infinite measure. Then, set $\wt\mu_n$ as the probability on $\calF_0(\N)$ determined by
\[
\frac{d\wt\mu_n}{d\wt\mu^*}(F) = \frac{\inddd{F\cap \{1,\dots,n\}\ne\emptyset}}{\wt\mu^*(\{F'\in\calF_0(\N): F'\cap\{1,\dots,n\}\ne\emptyset\})} = \frac{\inddd{F\cap\{1,\dots,n\}}}{w_n}, F\in\calF_0(\N).
\]
(We follow the convention to use the letter $F$ to denote a generic closed set. This is in conflict with the fact that we also use $F$ to denote the cumulative distribution function for the renewal. However, the notation $F$ standing for a closed set will not be used for the rest of the paper.)
{We let $R_n$ denote a random element in $\calF_{0,n}(\N):=\{F\in\calF_0(\N):F\cap\{1,\dots,n\}\ne\emptyset\}$ with law $\wt\mu_n$.} 
\color{black}
Moreover, it is immediately verified that the following lemma holds true.
\begin{Lem}\label{lem:Rm}
Let $R_n$ be as above. Then we have the following. 
 \begin{enumerate}[(i)]
  \item (Shift invariance) For all $k=2,\dots,n$, {$R_n\cap\{k,\dots\}-(k-1)\eqd R_n$}. In particular, $\proba(k\in R_n) = 1/w_n, k=1,\dots,n$. 
  \item (Markov/renewal property) $$\proba(\min(R_n\cap\{k+1,k+2,\dots\}) \le k+ j\mid k\in R_n) = F(j)$$ for all $j,k,n\in\N, n\ge j+k$.
 \end{enumerate}
(In the above, $A-k = \{\ell-k:\ell\in A\}$ and $\min (R_n\cap\{k+1,k+2,\dots\})$ is the smallest integer in the set of interest.)
\end{Lem}
For $k>n$, $\proba(k\in R_n)<1/w_n$ but this is not relevant for the rest of the paper.
By \eqref{eq:thinning}, we work with the following equivalent representation of \eqref{eq:series infty p}: 
\equh\label{eq:p>=1}
\pp{X_k}_{k=1,\dots,n}\eqd \pp{X_{n,k}}_{k=1,\dots,n} :=\pp{w_n^{2/\alpha}\sum_{0<i_1<  i_2}\frac{1}{ \Gamma_{i_1}^{1/\alpha}\Gamma_{i_2}^{1/\alpha}}\inddd{k\in R_{n,i_1,i_2} }}_{k=1,\dots,n},
\eque
where on the right-hand side the notation  is as in \eqref{eq:thinning}, and $
R_{n,i_1,i_2}  =  R_{n,i_1}\cap R_{n,i_2}$.

\subsection{Overview of the proof}
We prove Theorem \ref{thm:1}.
Recall that the goal of Theorem \ref{thm:1} is to show, under the assumption $d_n\to\infty, d_n = o(n)$,
\[
\xi_n  =\summ j1{k_n} \ddelta{\vv  X_{d_n,j}  /b_n, \ j/k_n    }\inddd{R_{n,i_1,i_2}\cap \cal{I}_{d_n,j}\neq \emptyset  }
\weakto \xi=\sif\ell1 \ddelta{  \vv{Q}_\ell (\Gamma_\ell/\theta)^{-1/\alpha},\ U_\ell     }.
\]
Recall
\[
\vvX_{d_n,j}:=\pp{X_{n,(j-1)d_n+1},\dots,X_{n,jd_n}},
\]
which is also viewed as a random element in $S$. 

We shall introduce {a few} intermediate processes as approximations. 
Let $m_n$ be an increasing sequence of positive integers. 
Eventually for Theorem \ref{thm:1} to hold we shall assume  $m_n\to\infty$ such that
\begin{equation}\label{eq:m_n ppc}
\frac{w_{n}^2}{n}\frac1{\log^{1/(1-\alpha)}n}\ll m_n\ll \frac{ w_n^2}{n}\log n.
\end{equation}
However, in a few steps below, we do not impose the above restrictions, but possibly a weaker constraint on the growth rate of $m_n$.
Set
 \equh {X}_{n,k}^*:= 
\pp{\frac{w_{n}}{m_n}}^{2/\alpha}
 \sum_{1\le i_1<i_2\le m_n}  V^*_{i_1,i_2}  \inddd{k\in R_{n,i_1,i_2}}  \label{eq:X star whole}, \ k=1,\ldots,n,
\eque
with 
\[
V_{i_1,i_2}^* \equiv V_{m_n,i_1,i_2}^*:=\pp{\frac{\Gamma_{m_n+1}^2}{\Gamma_{i_1}\Gamma_{i_2}}}^{1/\alpha}, \ i_1,i_2\in \N,\  i_1<i_2.
\]
Write
\[
{\vv X}_{d_n,j}^* :=  \pp{X_{n,(j-1)d_n+1}^*,\dots,X_{n,jd_n}^*},   \quad j=1,\ldots,k_n.
\]
Set also
\begin{equation}\label{eq:Q emp}
\vv Q_{n,d_n,j}^{(i_1,i_2)}:=\pp{\inddd{(j-1)d_n+k\in R_{n,i_1,i_2}}}_{k=1,\dots,d_n}, j=1,\dots,k_n.
\end{equation}
These and $\vv X_{d_n,j}$ are all defined as $d_n$-dimensional random vectors, and we use the same notation to represent the corresponding elements in $\wt\ell_0$.
Introduce the following sequence:
\[
  c_n = \pp{\frac{ b_n^{\alpha} m_n^2}{w_{n}^2}}^{1/\alpha}.
\]
Note that the rate of $c_n$ depends on the choice of $m_n$.

Now, define   two intermediate approximation point processes as follows:
\begin{align*}
\xi_n^*&:=\summ j1{k_n} \ddelta{\vv X_{d_n,j}^*/b_n, \ j/k_n    } \inddd{R_{n,i_1,i_2}\cap \cal{I}_{d_n,j}\neq \emptyset  },\\
\xi_n^\circ & :=\summ j1{k_n} \sum_{1\le i_1<i_2 \le m_n}
 \delta_{ \left( c_n^{-1} V_{i_1,i_2}^* \vv Q_{n,d_n,j}^{(i_1,i_2)}   , \ j/k_n   \right) } \inddd{R_{n,i_1,i_2}\cap \cal{I}_{d_n,j}\neq \emptyset  }.
\end{align*}

{\bf The proof proceeds as follows.}  Note that it suffices to show (see \citep[Theorem 7.1.17]{kulik20heavy}) that
\[
\xi_n(f)\weakto \xi(f)
\]
for all $f\in\Clipb(S\times[0,1])$, the class of bounded Lipschitz functions on $S\times [0,1]$ (see Section \ref{sec:topo} on topological background on $S$). We achieve so in the following three steps.
\begin{enumerate}[(i)]
\item (De-aggregation)
We first show in Proposition \ref{prop:5.2} that
\[\limn\esp \left|  \xi_n^*(f)- \xi_n^\circ(f) \right|=0,
\]
  for all  $f\in\Clipb(S\times[0,1])$, under a weaker condition: for some $\delta>0$,
  \equh\label{eq:m_n ppc 2}
\frac{w_{n}^2}{n}\frac1{\log^{1/(1-\alpha)}n}\ll m_n\ll\pp{  \frac{w_n}{d_n^\beta}\log n\wedge \frac{w_n}{n^\delta}}.
\eque
Actually we do not need the upper constraint in view of \eqref{eq:m_n ppc}. We note it here because the upper bound is involved in our analysis of the mesoscopic limit in Section \ref{sec:single}.
  
\item (Poisson approximation) Next we show in Proposition \ref{prop:5.3} that under \eqref{eq:m_n ppc 2} and $m_n\ll w_n^2\log n/n$ (that is, under \eqref{eq:m_n ppc}), we have as $n\rightarrow\infty$
\[
\xi_{n}^\circ\Rightarrow  \xi.
\]
We mention that the upper constraint $m_n\ll w_n^2\log n/n$ is needed only for the second-moment control for the Poisson approximation in \eqref{eq:bn2 bound} below. 
 
\item (Truncation approximation)
In Proposition \ref{prop:5.4}, we show that under \eqref{eq:m_n ppc},
\[
\limn\proba\pp{\abs{\xi_n(f) - \xi^*_n(f)}>\eta} = 0, \mfa \eta>0.
\]
\color{magenta}
\end{enumerate}

\subsection{De-aggregation}\label{sec:de-agg}
 Set
\[
J_n(k):= \{(i_1,i_2)\in \{1,\ldots,m_n\}^2:\  i_1<i_2, \  k\in R_{n,i_1,i_2} \}, \quad k=1,\ldots,n.
\]
As long as $m_n \ll w_nn^{-\delta}$ for some $\delta>0$, we may fix $p^*$ satisfying
\equh\label{eq:p*2}
p^* = \binom{\what p}2,\quad \text{such that } \limn \frac{nm_n^{\what p+1}}{w_n^{\what p+1}}=0.
\eque
Define 
\[
\Omega_{n,1}:=\{ |J_n(k)|\le  p^* \text{ for all }k =1,\ldots,n\}.
\]
Write $\calW_{m_n} :=\{(i_1,i_2): 1\le i_1<i_2\le m_n\}$. Then for $\epsilon>0$, introduce
\[
E^*_{n,i_1,i_2,j}(\epsilon):=\ccbb{ \frac{V_{i_1,i_2}^*}{c_n}>\epsilon,\  R_{n,i_1,i_2}\cap \cal{I}_{d_n,j}\neq \emptyset  },
\]
and 
\[
\Omega_{n,2}(\epsilon):= \ccbb{N_n(\epsilon)\le 1}\]
with
\begin{align}\label{eq:Nn}
 N_n(\epsilon):=\sum_{j=1}^{k_n}  \sum_{\substack{(i_1,i_2),(i_1',i_2')\in\calW_{m_n}\\ (i_1,i_2)\neq (i_1',i_2')}} \ind_{E^*_{n,i_1,i_2,j}(\epsilon)}\ind_{E^*_{n,i_1',i_2',j}(\epsilon)}.
\end{align}
\begin{Lem}\label{lem:5.1'}
Under the assumptions $d_n = o(n)$ and \eqref{eq:m_n ppc 2},
we have
\begin{equation}\label{eq:Omega_n(epsilon)}
\limn \proba(\Omega_{n,1}^c) = \limn\proba(\Omega_{n,2}(\epsilon)^c)=0.
\end{equation}
\end{Lem}
\begin{proof}
For $\Omega_{n,1}$, we have
\begin{align*}
\proba(\Omega_{n,1}^c)
  \le  \sum_{k=1}^{n}  \binom{m_n}{\what p+1}w_{n}^{-\what p-1} 
  \le    C \frac{nm_{n}^{\what p+1}}{w_n^{\what p+1}}\rightarrow 0.
\end{align*} 
For $\Omega_{n,2}(\epsilon)$, we have
\[
\proba((\Omega_{n,2}(\epsilon))^c)=\proba(N_n(\epsilon)\ge 1)\le \esp N_n(\epsilon),
\]
and hence it suffices to show $\limn \esp N_n(\epsilon)=0$. 

Introduce 
\[
V_{i_1,i_2}:=(U_{i_1}U_{i_2})^{-1/\alpha},  \ i_1,i_2\in \N,\  i_1<i_2,
 \]
 with $\{U_i\}_{i\in \N}$ as i.i.d.~uniform random variables which are independent from everything else, and accordingly
\[
 E_{n,i_1,i_2,j}(\epsilon):=\ccbb{ \frac{V_{i_1,i_2}}{c_n}>\epsilon,\  R_{n,i_1,i_2}\cap \cal{I}_{d_n,j}\neq \emptyset  }.
\]
 Note that for each pair $(i_1,i_2)$ fixed,  $E_{n,i_1,i_2,j}(\epsilon)$ and $E^*_{n,i_1,i_2,j}(\epsilon)$ are not identically distributed. However, recall
 \[
 \pp{\frac{\Gamma_1}{\Gamma_{m_n+1}},\cdots,\frac{\Gamma_{m_n}}{\Gamma_{m_n+1}},\Gamma_{m_n+1}}\EqD \pp{U_{1:m_n},\ldots,U_{m_n:m_n},\Gamma_{m_n+1}},
 \] where $U_{1:m_n}<\cdots<U_{m_n:m_n}$ are the order statistics of $\{U_i\}_{i=1,\dots,m_n}$. Since the sum $N_n$ {in \eqref{eq:Nn}} is invariant with a permutation of the index set $\{i:\ 1\le i\le m_n\}$,   we see that the law of $N_n$ will not change if all the sets $E_{n,i_1,i_2,j}^*(\epsilon)$ involved are replaced by $E_{n,i_1,i_2,j}(\epsilon)$ accordingly.  In particular,
\begin{align}
\esp N_n(\epsilon) &
=\sum_{j=1}^{k_n}  \sum_{\substack{(i_1,i_2),(i_1',i_2')\in\calW_{m_n}\\ (i_1,i_2)\neq (i_1',i_2')}} \esp\pp{\ind_{E^*_{n,i_1,i_2,j}(\epsilon)}\ind_{E^*_{n,i_1',i_2',j}(\epsilon)}}\nonumber\\
&
=  \sum_{j=1}^{k_n}  \sum_{\substack{(i_1,i_2),(i_1',i_2')\in\calW_n \\ \{i_1,i_2\} \cap \{i_1',i_2'\}=\emptyset}}   \esp \ind_{E_{n,i_1,i_2,j}(\epsilon)}\esp \ind_{E_{n,i_1',i_2',j}(\epsilon)}\label{eq:EN}\\
& \quad + \sum_{j=1}^{k_n}  \sum_{\substack{(i_1,i_2),(i_1',i_2')\in\calW_n\\ |\{i_1,i_2\}\cap \{i_1',i_2')\}|=1}} \esp\pp{ \ind_{E_{n,i_1,i_2,j}(\epsilon)}\ind_{E_{n,i_1',i_2',j'}(\epsilon)} }.\nonumber \\
& =k_n \binom{m_n}4 (\esp\ind_{E_{n,1,2,1}(\epsilon)})^2 + 3k_n\binom{m_n}3 \esp \pp{\ind_{E_{n,1,2,1}(\epsilon)}\ind_{E_{n,1,3,1}(\epsilon)}}  =: \mathrm{I}_n+ \mathrm{II}_n.\nonumber
\end{align}
\color{black}
We start with some estimates. 
First,  since
\[
\proba(V_{1,2}>x)= \proba(U_1U_2<x^{-\alpha})\sim  \alpha   x^{-\alpha} \log x,
\]
  we have   for $y>0$,
\begin{equation}\label{eq:V_1,2 tail}
\proba(V_{1,2}>yc_n) \sim  \frac{2  w_{n}^2}{m_n^2 (n\log n) }   \log\pp{\frac{nm_n^2}{w_n^{2}}}  y^{-\alpha} \sim \frac{2\log(nm_n^2/w_n^2)}{\log n}\frac{w_n^2}{m_n^2 n} y^{-\alpha}.
\end{equation}
We write the dependence on $m_n$ explicitly for later discussions. 
We also have (see Lemma \ref{Lem:conv Q})
\[
\proba\pp{R_{n,1,2}\cap\calI_{d_n,1}\ne\emptyset} \sim \mathsf q_{F,2} \frac{d_n}{w_n^2}.
\]
This and \eqref{eq:V_1,2 tail} then imply as $n\rightarrow\infty$,
\begin{align}
\rho_n(\epsilon) & :=\proba\pp{{E_{n,1,2,1}(\epsilon)}} = \proba(c_n^{-1}V_{1,2}>\epsilon)\proba\pp{R_{n,1,2}\cap \cal{I}_{d_n,1}\neq \emptyset}\nonumber\\
& \le C\frac{\log(nm_n^2/w_n^2)}{\log n}\mathsf q_{F,2}  \epsilon^{-\alpha} \frac{w_n^2}{m_n^2 n} \frac{ d_n}{w_n^2}
\le  \frac{ C}{\epsilon^\alpha m_n^2 k_n}.\label{eq:rho_n upper}
\end{align}
By \eqref{eq:rho_n upper},  we have
\[ 
\mathrm{I}_n\le  C k_n m_n^4 \rho_n^2(\epsilon)\le C  k_n^{-1}  \rightarrow 0
\] 
as $n\rightarrow\infty$, where we used the assumption $n/d_n\to\infty$. We also have
\begin{align*}
\proba(R_{n,1,2}\cap\calI_{d_n,1}\ne\emptyset,R_{n,1,3}\cap\calI_{d_n,1}\ne\emptyset)& \le \summ k1{{d_n}}\summ{k'}1{{d_n}}\proba(k\in R_{n,1,2}\cap\calI_{d_n,1},k'\in R_{n,1,3}\cap\calI_{d_n,1}) \\
&\le \frac2{w_{n}^3}\summ k1{d_n}\sum_{k'=k}^{d_n}u(k'-k)\le C \frac{d_n^{1+\beta}}{w_{n}^3},
\end{align*}
where we have used \eqref{eq:u(n) asymp} in the last inequality,  and hence
\[
\mathrm{II}_n\le C k_n m_n^3 \frac{d_n^{1+\beta}}{w_n^3}c_n^{-\alpha} \sim C \frac {m_n d_n^\beta}{w_n \log n}.
\]
For the last expression to tend to zero, we need $m_n\ll w_n\log n/d_n^\beta$, which is exactly the assumption on the upper bound of $m_n\to\infty$ in \eqref{eq:m_n ppc}. We have thus proved \eqref{eq:Omega_n(epsilon)}.
\end{proof}

\begin{Prop}\label{prop:5.2}
Under the assumptions in Lemma \ref{lem:5.1'}, we have
\[\limn\esp \left|  \xi_n^*(f)- \xi_n^\circ(f) \right|=0
\]  for all  $f\in\Clipb(S\times[0,1])$.
\end{Prop}
\begin{proof}
With $\epsilon>0$ and $p^*$  as in \eqref{eq:p*2}, define \[E_{p^*\epsilon}:=S_{p^*\epsilon}\times [0,1] ,\] where $S_{p^*\epsilon}=\{\vv z\in S :\  \|\vv z\|_\infty>  p^* \epsilon \}$. 
Now suppose $f\in \Clipb(S\times[0,1])$ whose support is within $E_{p^*\epsilon}$. 
Recall 
\equh\label{eq:xi_n*}
\xi_n^*(f) = \summ j1{k_n}f\pp{\frac{\vvX_{d_n,j}^*}{b_n}, \frac j{k_n}} \qmand \xi_n^\circ(f) = \summ j1{k_n}\sum_{1\le i_1<i_2\le m_n}f\pp{\frac{V_{i_1,i_2}^*\vv Q_{n,d_n,j}\topp{i_1,i_2}}{c_n},\frac j{k_n}}.
\eque
For fixed $j=1,\dots,k_n$, we examine the difference
\[
\mathfrak d_{n,j}:=f\pp{\frac{\vvX_{d_n,j}^*}{b_n}, \frac j{k_n}} - \sum_{1\le i_1<i_2\le m_n}f\pp{\frac{V_{i_1,i_2}^*\vv Q_{n,d_n,j}\topp{i_1,i_2}}{c_n},\frac j{k_n}}.
\]
Observe that on the event $\Omega_{n,1}\cap \Omega_{n,2}(\epsilon)$,   there are two cases.
\begin{enumerate}[(i)]
\item None of $1\le i_1<i_2\le m_n $ satisfies both $R_{n,i_1,i_2}\cap \cal{I}_{d_n,j}\neq \emptyset$ and   $   c_n^{-1} V^*_{i_1,i_2}>\epsilon $. In this case,  
\[
\frac{\nn{\vvX_{d_n,j}^*}_\infty}{b_n}\le p^*\epsilon \qmand \frac1{c_n}\max_{(i_1,i_2)\in\calW_n}\nn{
 V_{i_1,i_2}^*\vv Q_{n,d_n,j}\topp{i_1,i_2}}_\infty\le \epsilon.
\]
Therefore,  $\mathfrak d_{n,j} = 0$.

\item Exactly one of $(i_1,i_2)\in\calW_n$ is such that the event $E^*_{n,i_1,i_2,j}(\epsilon) = \{R_{n,i_1,i_2}\cap \cal{I}_{d_n,j}\neq \emptyset,  c_n^{-1} V^*_{i_1,i_2}>\epsilon\}$ holds. Denote this (random) pair by $(\what i_1,\what i_2)$. Then,  
{
\[
\frac1{b_n}\vvX_{d_n,j}^*  = \frac1{c_n}V_{\what i_1,\what i_2}^*+ \sum _{\substack{(\iota_1,\iota_2)\in \what I}}\frac1{c_n}V_{\iota_1,\iota_2}^*,
\]
where $\what I = \{(i_1,i_2)\in\calW_n:R_{n,i_1,i_2}\cap\calI_{d_n,j}\ne\emptyset\}\setminus  \{(\what i_1,\what i_2)\}$} is a random subset of $\calW_{m_n}$ with at most $p^*-1$ elements, and for each $(\iota_1,\iota_2)\in\what I$ we have $V^*_{\iota_1,\iota_2}/c_n\le \epsilon$.
 Then by the  Lipchitz property of $f$, we have
\[
\abs{\mathfrak d_{n,j}} = \abs{f\pp{\frac1{c_n}V^*_{\what i_1,\what i_2}+ \sum _{(\iota_1,\iota_2)\in \what I}\frac1{c_n}V^*_{\iota_1,\iota_2},\frac j{k_n}} - f\pp{\frac1{c_n}V^*_{\what i_1,\what i_2},\frac j{k_n}}}\le C(p^*-1)\epsilon.
\]
\end{enumerate}
In summary, we have for every $j=1,\dots,k_n$, 
\begin{align*}
\esp\pp{|\mathfrak d_{n,j}|;\Omega_{n,1}\cap\Omega_{n,2}(\epsilon)} &= \sum_{1\le i_1<i_2\le m_n}\esp\pp{|\mathfrak d_{n,j}|;E^*_{n,i_1,i_2,j}(\epsilon), \Omega_{n,1}\cap\Omega_{n,2}(\epsilon)}\\
&\le C(p^*-1)\epsilon \times \binom {m_n}2\proba(E_{n,1,2,1}(\epsilon)).
\end{align*}
Therefore,  we have
 by \eqref{eq:rho_n upper}   that
\begin{align*}
\limsupn\esp \left|  \xi_n^*(f)- \xi^\circ_n(f) \right|& = \limsupn \esp\pp{|\xi_n^*(f)-\xi^\circ_n(f)|; \Omega_{n,1}\cap\Omega_{n,2}(\epsilon)} \\
& \quad + C\limsupn  \proba((\Omega_{n,1}\cap \Omega_{n,2}(\epsilon))^c)\\
&\le C  (p^ *-1) \epsilon  \limsup_{n\rightarrow\infty} k_n \binom{m_n}2\rho_n(\epsilon) \\&\le   C   \epsilon^{1-\alpha}.
\end{align*}
Letting $\epsilon\downarrow 0$ and noting $\alpha\in (0,1)$, we conclude the proof.
 \end{proof}

\subsection{Poisson approximation}
\begin{Prop}\label{prop:5.3}
Under the assumption 
$d_n\to\infty, d_n = o(n)$ and \eqref{eq:m_n ppc}, we have  the weak convergence of point processes
\[
\xi_{n}^\circ\Rightarrow  \xi,
\]
as $n\rightarrow\infty$.
\end{Prop}
\begin{proof}
We prove a stronger result:
with
\begin{equation}\label{eq:eta_n}
\eta_n=\summ j1{k_n} \sum_{1\le i_1<i_2 \le m_n}
 \delta_{ \left( c_n^{-1} V^*_{i_1,i_2}, \ \vv Q_{n,d_n,j}^{(i_1,i_2)}   , \ j/k_n   \right) }\inddd{R_{n,i_1,i_2}\cap \cal{I}_{d_n,j}\neq \emptyset  },
\end{equation}
we have as $n\rightarrow\infty$,
\equh\label{eq:eta}
\eta_n \Rightarrow \eta  :=\sif\ell1 \ddelta{ \theta^{1/\alpha} \Gamma_\ell^{-1/\alpha},\  \vv{Q}_\ell ,\ U_\ell     }.
\eque
The desired result follows immediately from the above. Indeed, define
\equh\label{eq:T}
T: (0,\infty]\times S^* \times [0,1]\mapsto  S    \times [0,1], \ (y,\vv{z},u)\mapsto (y\vv{z},u) ,
\eque
where $S^*=\{[\vv z]\in S:  \ \|\vv z\|_\infty= \wt{d}([\vv z],[\vv 0]) =1\}$ is equipped with the subspace topology. It is elementary to verify that $T$ is continuous. In addition, the preimage $T^{-1}K$ is bounded for any bounded $K\subset S    \times [0,1]$.   Then apply a generalization of  \cite[Proposition 5.5]{resnick07heavy}, where the role of compact sets is replaced by bounded sets: the boundedness of  $ S \times [0,1]$ is formed by subsets $B$  satisfying $B\subset A\times [0,1]$, with $A\subset S$ and $\inf_{[\vv x]\in A}\wt{d}([\vv x],[\vv 0])>0$, and the boundedness of  $(0,\infty]\times S^* \times [0,1]$ is formed by subsets $B$ satisfying $B\subset (\delta,\infty]\times S^*\times [0,1]$ for some $\delta>0$.

It remains to prove \eqref{eq:eta}. 
Take disjoint $E_s= (y_{s,0},y_{s,1}]\times B_s\times I_s $, $s=1,\ldots,d$, where $B_s$ is a Borel set in $\wt{\ell}_0\setminus\{ \wt{ \vv 0 } \}$, and $I_s$ is a subinterval of $[0,1]$ of the form $(a_s,b_s]$ or $[0,b_s]$  $s=1,\ldots,d$.  
Since sets of the form $E_s$ form a dissecting semiring (see \cite[Section 1.1]{kallenberg17random}), in view of \cite[Theorem 4.11]{kallenberg17random}, it suffices to show the joint convergence
\begin{equation}\label{eq:joint poisson}
\left( \eta_n(E_s)\right)_{s=1,\ldots,d} \Rightarrow  \left( \eta(E_s)\right)_{s=1,\ldots,d} \mmas n\to\infty.
\end{equation}
We prove \eqref{eq:joint poisson} by the two-moment method for the Poisson approximation (see \citep{arratia89two} and \citep[Theorem 2.3]{penrose03random}).
Write 
\[
 \eta_n(E_s)= \summ j1{k_n} \sum_{1\le i_1<i_2 \le m_n} \chi_{n,s}(i_1,i_2,j)
\]
where
\[
\chi_{n,s}(i_1,i_2,j)=
 \inddd{ c_n^{-1} V^*_{i_1,i_2}\in (y_{s,0},y_{s,1}], \ \vv Q_{d_n,j}^{(i_1,i_2)} \in B_s   , \ j/k_n\in I_s   } \inddd{ R_{n,i_1,i_2}\cap \cal{I}_{d_n,j}\neq \emptyset  }
\]
Note that by choosing $\epsilon=\min_{s=1,\dots,d} y_{s,0}>0$, 
\begin{equation}\label{eq:chi replace}
\chi_{n,s}(i_1,i_2,j)\le \ind_{E^*_{n,i_1,i_2,j}(\epsilon)}= \inddd{ c_n^{-1} V^*_{i_1,i_2}>\epsilon,\  R_{n,i_1,i_2}\cap \cal{I}_{d_n,j}\neq \emptyset  }.
\end{equation}
Write
\[
\calV_{m_n,d_n}:=\{(i_1,i_2,j):\  1\le i_1<i_2\le m_n,\  j=1,\dots,k_n\},
\]
and  $(i_1,i_2)\sim (i_1',i_2')$ if $\{i_1,i_2\}\cap \{i_1',i_2'\}\neq\emptyset$. Note that when $(i_1,i_2)\not\sim(i_1',i_2')$, then $\chi_{n,s}(i_1,i_2,j)$ and $\chi_{n,s}(i_1',i_2',j')$ are independent Bernoulli random variables for all $j,j'$.
In order to show \eqref{eq:joint poisson}, by \citep[Theorem 2.3]{penrose03random}, it suffice to show 
\equh\label{eq:mean conv}
\esp \eta_n(E_s)   \rightarrow \esp \eta(E_s), \quad s=1,\ldots,d,
\eque
\begin{align}\label{eq:b_{n,1}}
b_{n,1}(\epsilon) &:=\sum_{\substack{(i_1,i_2,j), (i_1',i_2',j')\in \cal{V}_{m_n,d_n}\\ (i_1,i_2)\sim  (i_1',i_2')}} \esp\pp{ \ind_{E^*_{n,i_1,i_2,j}(\epsilon)}  \ind_{E^*_{n,i_1',i_2',j'}(\epsilon)} }\rightarrow 0,\\
\label{eq:b_{n,2}}
b_{n,2}(\epsilon)& :=\sum_{\substack{(i_1,i_2,j)\ne (i_1',i_2',j')\in \cal{V}_{m_n,d_n}\\ (i_1,i_2)\sim  (i_1',i_2')}} \esp \pp{\ind_{E^*_{n,i_1,i_2,j}(\epsilon)} \ind_{E^*_{n,i_1',i_2',j'}(\epsilon)}} \rightarrow 0,
\end{align}
  as $n\rightarrow\infty$. {Similarly as in the argument around \eqref{eq:EN},  the sets $E_{n,i_1,i_2,j}^*(\epsilon)$ can be replaced by $E_{n,i_1,i_2,j}(\epsilon)$ without changing the values of $b_{n,1}(\epsilon)$ and $b_{n,2}(\epsilon)$, which we shall implement below. }
Note that when controlling the second moments that are necessary for applying \citep[Theorem 2.3]{penrose03random}, instead of computing the corresponding moments of $\chi_{n,s}(i_1,i_2,j)$ above, it suffices to establish \eqref{eq:b_{n,1}} and \eqref{eq:b_{n,2}} due to the upper bound in \eqref{eq:chi replace}. Then, the above desired convergence has already been established in \citep{bai24phase}. We summarize these estimates in the following.

For \eqref{eq:mean conv}, by exchangeability and recalling \eqref{eq:V_1,2 tail} we have
\begin{align*}
\esp \eta_{n}(E_s)&=k_n \binom{m_n}2 \proba\pp{c_n^{-1}V_{1,2}\in (y_{s,0},y_{s,1} ]}
 \sum_{j/k_n \in I_s}\proba\pp{\vv Q_{n,d_n,j}^{(1,2)}\in B_s, \vv Q_{n,d_n,j}\topp{1,2}\ne\vv 0}\\
 & = k_n \binom{m_n}2 \proba\pp{c_n^{-1}V_{1,2}\in (y_{s,0},y_{s,1} ]}
\\
& \quad\times \proba\pp{\vv Q_{n,d_n,1}^{(1,2)}\in B_s\mmid R_{n,1,2}\cap\calI_{d_n,1}\ne\emptyset} \proba \pp{R_{n,1,2}\cap\calI_{d_n,1}\ne\emptyset}  \pp{ \sum_{j/k_n \in I_s}1}.
\end{align*}
Now we have a more precise estimate than \eqref{eq:rho_n upper}:
\begin{align}
\rho_n(\epsilon) & = \proba(c_n^{-1}V_{1,2}>\epsilon)\proba\pp{R_{n,1,2}\cap \cal{I}_{d_n,1}\neq \emptyset} \sim 2\frac{\log(nm_n^2/w_n^2)}{\log n}\mathsf q_{F,2}  \epsilon^{-\alpha} \frac{w_n^2}{m_n^2 n} \frac{ d_n}{w_n^2}\label{eq:rho_n-a}\\
&\sim 2 \vartheta_1 \epsilon^{-\alpha} \frac{1}{m_n^2 k_n},\nonumber
\end{align}
where the last relation follows from \eqref{eq:m_n ppc}. 
We also have $\limn\proba\spp{\vv Q_{n,d_n,1}^{(1,2)}\in B_s\mid R_{n,1,2}\cap\calI_{d_n,1}\ne\emptyset}= \proba(\vv Q\in B_s)$ by Lemma \ref{Lem:conv Q}. 
Therefore,
\equh
\esp \eta_{n}(E_s)\sim 
   \vartheta_1
\left(y_{s,0}^{-\alpha}-y_{s,1}^{-\alpha}\right)  \proba\pp{\vv Q \in B_s } \Leb(I_s),\label{eq:mean}
\eque
where $\Leb$ stands for the Lebesgue measure.
We have proved \eqref{eq:mean conv}.

For  \eqref{eq:b_{n,1}},  by exchangeability, an elementary combinatorial calculation, and \eqref{eq:rho_n upper}, we have
\begin{align*}
b_{n,1}(\epsilon)\le  C   k_n^2 m_n^3      \left( \frac{1}{m_n^2 k_n}\right)^2\le C \frac{1}{m_n}\rightarrow 0.
\end{align*}
It remains to verify \eqref{eq:b_{n,2}} with $E_{n,i_1,i_2,j}(\epsilon)$ replaced by $E^*_{n,i_1,i_2,j}(\epsilon)$.
First, for $(i_1,i_2,j)\sim (i_1',i_2',j')$ and $(i_1,i_2,j)\neq  (i_1',i_2',j')$, we have
\begin{align*}
&\esp \pp{
\ind_{E_{n,i_1,i_2,j}(\epsilon)} \ind_{E_{n,i_1',i_2',j'}(\epsilon)}}\\= &
\begin{cases}
\proba(c_n^{-1 }V_{1,2}>\epsilon, c_n^{-1} V_{2,3}>\epsilon) \proba(R_{n,1,2}\cap \cal{I}_{n,j} \neq \emptyset, R_{n,2,3}\cap \cal{I}_{n,j'} \neq \emptyset   ),  &\hspace{-.2cm} |\{i_1,i_2\}\cap\{i_1',i_2'\}|=1. 
\\
 \proba(c_n^{-1 }V_{1,2}>\epsilon) \proba(R_{n,1,2}\cap \cal{I}_{n,j} \neq \emptyset, R_{n,1,2}\cap \cal{I}_{n,j'} \neq \emptyset   ), &  \hspace{-.2cm}(i_1,i_2)=(i_1',i_2'), j\neq j'.
\end{cases}
\end{align*}
Set $\rho_{n,d_n}\topp1 $ and $\rho_{n,d_n}\topp2 $ as in Lemma \ref{lem:rho}.
Then it follows from a  simple combinatoric consideration,  
the observation 
\[
\proba(U_1 U_2< x, U_{2}U_{3}<x) =\int_0^1\proba(U_1 < x/u)^2 du =\int_0^1   [(x/u) \wedge 1]^2 du  \sim 2x  \text{ as }x\rightarrow 0,
\]
Lemma \ref{lem:rho} and \eqref{eq:V_1,2 tail},     that 
\begin{align}\label{eq:bn2 bound}
b_{n,2}(\epsilon)& \le  C m_n^3 \proba\pp{c_n^{-1 }V_{1,2}>\epsilon, \ c_n^{-1} V_{2,3}>\epsilon} \rho_{n,d_n}\topp1 + m_n^2\proba \pp{c_n^{-1 }V_{1,2}>\epsilon} \rho_{n,d_n}\topp2\notag\\
& \le    Cm_n^3  c_n^{-\alpha}  \frac{n^{1+\beta}}{w_n^3} +  C m_n^2 \frac{w_n^2}{m_n^2 n}       \frac{nd_n^{2\beta-1 } }{w_n^2} \le C  \frac{m_n n}{w_n^2\log n} + C d_n^{2\beta-1 }\rightarrow 0,
\end{align}
  since $m_n\ll { w_n^2\log n}/{n}$ from \eqref{eq:m_n ppc} and $\beta<1/2$.
\end{proof}
\subsection{Truncation}
We have seen that 
\[
\xi_n^*\weakto \xi
\]
as $n\to\infty$, following Propositions \ref{prop:5.2} and \ref{prop:5.3}. It remains to approximate $\xi_n$ and $\xi_n^*$. 

\begin{Prop}\label{prop:5.4}
For all $f\in \Clipb(S\times[0,1])$, under \eqref{eq:m_n ppc},
\[
\limn\proba\pp{\abs{\xi_n(f) - \xi_n^*(f)}>\eta} = 0, \mfa \eta>0.
\]
\end{Prop}

For this purpose we introduce another layer of approximation. Define
\begin{align}
\wt{X}_{n,k}  & :=  
 w_{n}^{2/\alpha}\sum_{1\le i_1<i_2\le m_n}   \frac{1}{\Gamma_{i_1}^{1/\alpha}\Gamma_{i_2}^{1/\alpha}}\inddd{k\in  R_{n,i_1,i_2}} \label{eq:X tilde}\\
 & =\pp{\frac{w_n} {\Gamma_{m_n+1}}}^{2/\alpha} 
 \sum_{1\le i_1<i_2\le m_n}  \pp{\frac{\Gamma_{m_n+1}^2}{\Gamma_{i_1}\Gamma_{i_2}}}^{1/\alpha}    \inddd{k\in R_{n,i_1,i_2}},\  k=1,\dots,n.\nonumber
 \end{align}
 Write
\[
\wt\vvX_{d_n,j}:=\pp{\wt X_{n,(j-1)d_n+1},\dots,\wt X_{n,jd_n}},
   \quad j=1,\ldots,k_n.
\]
and 
\[
\wt\xi_n:=\summ j1{k_n} \ddelta{\wt\vvX_{d_n,j}/b_n, \ j/k_n    } \inddd{R_{n,i_1,i_2}\cap \cal{I}_{d_n,j}\neq \emptyset  }.
\]
Fix $f\in\Clipb(S\times[0,1])$. Let $\epsilon>0$ be such that $f(\vv x, \cdot)\equiv 0$ if $\|\vv x\|_\infty<\epsilon$. To complete the proof, in view of \cite[Theorem 7.1.17]{kulik20heavy}, it suffices to show that 
for any $\eta>0$, 
\[
\limn\proba\pp{ |\xi_n(f)-\wt{\xi}_n(f)|>\eta}=0,
\]
and
\[
\limn\proba\pp{ |\wt{\xi}_n(f)- \xi^*_n(f)|>\eta}=0.
\]
We use arguments similar to the proof of \cite[Theorem 7.4.3]{kulik20heavy}. The key is the following lemma.

\begin{Lem}
We have, for all $\delta\in(0,\epsilon/2)$,
\begin{multline}\label{eq:bridge1}
 \proba\pp{ |\xi_n(f)-\wt{\xi}_n(f)|>\eta} \\
\le    \proba\pp{ \max_{k=1,\ldots,n} |X_{n,k}-\wt{X}_{n,k}| >b_n \delta }+\proba\pp{  \wt{\xi}_n \left(  \ccbb{\|\vv x\|_\infty>\frac\epsilon2}  \times [0,1] \right)>\frac{C\eta}{\delta} },
\end{multline}
and
\begin{multline}\label{eq:bridge2}
 \proba\pp{ |\wt\xi_n(f) -\xi^*_n(f)|>\eta} \\
\le    \proba\pp{ \max_{k=1,\ldots,n} |\wt X_{n,k}-X^*_{n,k}| >b_n \delta }+\proba\pp{  \xi^*_n \left(  \ccbb{\|\vv x\|_\infty>\frac\epsilon2}  \times [0,1] \right)>\frac{C\eta}{\delta} },
\end{multline}
where the constant $C>0$ depends only on $f$.
\end{Lem}
\begin{proof}
We only show \eqref{eq:bridge1}, and a similar argument yields \eqref{eq:bridge2}. 
Write, for all $\eta,\delta>0$, \begin{multline}\label{eq:bridge1'}
 \proba\pp{ |\xi_n(f)-\wt{\xi}_n(f)|>\eta} \\
\le   \proba\pp{ \max_{k=1,\ldots,n} |X_{n,k}-\wt{X}_{n,k}| >b_n \delta }+\proba\pp{  |\xi_n(f)-\wt{\xi}_n(f)|>\eta,  \max_{k=1,\ldots,n} |X_{n,k}-\wt{X}_{n,k}| \le b_n \delta }. 
\end{multline}
We examine the second probability on the right-hand side above.
 Take any $\delta\in (0,\epsilon/2)$. 
 Note that  when  $\max_{k=1,\ldots,n} |X_{n,k}-\wt{X}_{n,k}| \le b_n \delta$, if for some $j=1,\dots,k_n$,  $\|\wt{\vvX}_{d_n,j}/b_n\|_\infty\le \epsilon/2$, then necessarily we have $\| \vvX_{d_n,j}/b_n\|_\infty<\delta+\epsilon/2<\epsilon$  and hence $f( \wt\vvX_{d_n,j}/b_n, j/k_n) = f( \vvX_{d_n,j}/b_n, j/k_n)=0$. That is, when restricting to $\{\max_{k=1,\ldots,n} |X_{n,k}-\wt{X}_{n,k}| \le b_n \delta\}$ and comparing the difference, we have,
 \begin{align*}
 \abs{\xi_n(f)-\wt\xi_n(f)}& \le \summ j1{k_n}\abs{f\pp{\frac{\vvX_{d_n,j}}{b_n},\frac j{k_n}} - f\pp{\frac{\wt\vvX_{d_n,j}}{b_n},\frac j{k_n}}}\\
 & \le \summ j1{k_n}C\frac{\snn{\vvX_{d_n,j}-\wt\vvX_{d_n,j}}_\infty}{b_n}\inddd{\nn{\wt\vvX_{d_n,j}}_\infty/b_n>\epsilon/2}\\
 & \le \summ j1{k_n}C\delta\inddd{\nn{\wt\vvX_{d_n,j}}_\infty/b_n>\epsilon/2} = C\delta \wt\xi_n(\{\nn\vvx_\infty>\epsilon/2\}\times[0,1]),
 \end{align*}
 where in the second inequality above we used the Lipschitz property. Plugging the above into \eqref{eq:bridge1'},  
we have proved \eqref{eq:bridge1}. 
\end{proof}

To complete the proof of Proposition \ref{prop:5.4}, it suffices to examine \eqref{eq:bridge1} and \eqref{eq:bridge2} and show that the upper bounds tend to 0. We start with \eqref{eq:bridge2}. 
Write 
\[
\vv X^*_{n} = (X^*_{n,1},\dots,X^*_{n,n})  \qmand \wt {\vv X}_{n}=(\wt X_{n,1},\dots,\wt X_{n,n}).
\]
\begin{Lem}\label{lem:2'}
For all $\epsilon>0$, under \eqref{eq:m_n ppc},
\[
\limn\proba\left( \frac1{b_n} \|\vvX^*_{n}   - \wt\vvX_{n}   \|_\infty>\epsilon  \right) = 0.
\]
\end{Lem}
\begin{proof}
Notice that we have
\[
\wt\vvX_{n} = \pp{\frac{m_n}{\Gamma_{m_n+1}}}^{2/\alpha} \vvX^*_{n},
\]
and therefore
 \begin{align*}  
 \proba\left( \frac1{b_n} \|\vvX^*_{n}   - \wt{\vv X}_{n}   \|_\infty>\epsilon  \right)& =  \proba\pp{\abs{\left(\frac{m_n}{\Gamma_{m_n+1}}\right)^{2/\alpha} - 1}\times \frac1{b_n}  \|\vvX^*_{n}\|_\infty >\epsilon } \\
& \le  \proba\pp{\abs{\left(\frac{m_n}{\Gamma_{m_n+1}}\right)^{2/\alpha} - 1}>\epsilon' }+\proba\pp{\frac {\epsilon'}{b_n} \|\vvX^*_{n}    \|_\infty >\epsilon },
\end{align*}
for all $\epsilon'>0$.
For the second term above, {by Propositions \ref{prop:5.2} and \ref{prop:5.3}, $\xi_n^*\weakto\xi$ and as a consequence} we have
\[ 
\proba\pp{\frac {\epsilon'}{b_n} \|\vvX^*_{n}    \|_\infty >\epsilon }\sim C (\epsilon/\epsilon')^{-\alpha},
\]
and also $(\epsilon/\epsilon')^{-\alpha}\to0$ as $\epsilon'\downarrow 0$.  For the first term above, we have
\begin{align*}
 \proba &\pp{\abs{\left(\frac{m_n}{\Gamma_{m_n+1}}\right)^{2/\alpha} - 1}>\epsilon' }\\
 &=\proba\pp{ \frac{m_{n}+1}{\Gamma_{m_n+1}}  > \frac{m_{n}+1}{m_n} (1+\epsilon')^{\alpha/2} } + 
 \proba\pp{ \frac{m_{n}+1}{\Gamma_{m_n+1}}  < \frac{m_{n}+1}{m_n} (1-\epsilon')^{\alpha/2} } \\
 &\le \proba\pp{ \abs{\frac{\Gamma_{m_n+1}}{m_{n}+1} -1 }>\gamma },
\end{align*}
for some $\gamma>0$ small enough (depending on $\epsilon'$). 
Note that $\Gamma_{m_n+1}$ is a   sum of $m_n+1$  i.i.d.~exponential variables, and hence
\begin{equation}\label{eq:gamma concentration}
\proba\pp{\abs{\frac{\Gamma_{m_{n}+1 }}{m_n+1}- 1}>\gamma} \le \frac{\esp(\Gamma_{m_n+1}-\esp\Gamma_{m_n +1})^{2q}}{(m_{n}+1)^{2q}\gamma^{2q}}  \le  \frac {C_q}{m_n ^q\gamma^{2q}}, \mfa q\in\N.
\end{equation}
Therefore, by taking $q$ large enough, the above decays faster than any polynomial rates (we shall only need $o(1)$ here, but $o(k_n\inv)$ later). 
The desired result now follows.\end{proof}

For the second term on the right-hand side of \eqref{eq:bridge2}
 since we have proved $\xi_n^*\weakto\xi$ as $n\rightarrow\infty$ in Propositions \ref{prop:5.2} and \ref{prop:5.3} (again under \eqref{eq:m_n ppc})
\[
\limn\proba\pp{ \xi^*_n \left(  \ccbb{\|\vv x\|_\infty>\frac\epsilon2}  \times [0,1] \right)>\frac{C\eta}{\delta} } = \proba\pp{   \xi  \left(  \ccbb{\|\vv x\|_\infty>\frac\epsilon2}  \times [0,1] \right)>\frac{C\eta}{\delta} }.
\]
Because $\xi$ is almost surely finite on $\{\|\vv x\|_\infty>\epsilon/2\}  \times [0,1]$, the right-hand side above tends to $0$  by letting $\delta\downarrow 0$. This completes the proof of $\limn\proba\spp{ |\wt\xi_n(f)-\xi^*_n(f)|>\eta}=0$. 

Next, we examine \eqref{eq:bridge1}. For the second term, it vanishes following the same argument as above using $\wt\xi_n\weakto\xi$. For the first term we prove it in the following lemma. This completes the proof of $\limn\proba\spp{ |\wt\xi_n(f)-{\xi}_n(f)|>\eta}=0$, and then the proof of Proposition \ref{prop:5.4}.

\begin{Lem}\label{Lem:remainder whole n}
Under the assumption
$m_n\gg  w_{n}^2/(n\log^{1/(1-\alpha)}n)$, we have for any $\delta>0$ that
\[
\limn \proba\pp{ \max_{k=1,\ldots,n} |X_{n,k}-\wt{X}_{n,k}| >b_n \delta }=0
\]
\end{Lem}
\begin{proof}First, choose a fixed number $m$ so that $m>2/\alpha$. Assume without loss of generality that $m_n>2/\alpha$.   Introduce
\begin{align*}
A_{n,1}  & := \frac{w_{n}^{2/\alpha}}{b_n}\max_{k=1,\ldots,n}\sum_{1\le i_1 \le m,  m_n<i_2}   \frac{1}{\Gamma_{i_1}^{1/\alpha}\Gamma_{i_2}^{1/\alpha}}\inddd{k\in   R_{n,i_1,i_2}},
\\
A_{n,2}  & := \frac{w_{n}^{2/\alpha}}{b_n}\max_{k=1,\ldots,n}\sum_{m< i_1 \le m_n<i_2}   \frac{1}{\Gamma_{i_1}^{1/\alpha}\Gamma_{i_2}^{1/\alpha}}\inddd{k\in   R_{n,i_1,i_2}},
\\
A_{n,3} &:=  \frac{w_{n}^{2/\alpha}}{b_n} \max_{k=1,\ldots,n}\sum_{  m_n<i_1<i_2 }   \frac{1}{\Gamma_{i_1}^{1/\alpha}\Gamma_{i_2}^{1/\alpha}}\inddd{k\in   R_{n,i_1,i_2}}. 
\end{align*}
It suffices to show 
\[
A_{n,i}\to 0 \mbox{ in probability, } i=1,2,3,
\]
as $n\to\infty$. 
We first deal with $A_{n,2}$. By bounding the maximum of non-negative numbers with their sum and an inequality     $\esp\spp{\Gamma_{i_1}^{-1/\alpha} \Gamma_{i_2}^{-1/\alpha}}\le C i_1^{-1/\alpha}i_2^{-1/\alpha}$ which holds for large enough $i_1,i_2$ \citep[the inequality (3.2)]{samorodnitsky89asymptotic}, we have
\begin{align}
\esp A_{n,2}  & \le  \frac{w_{n}^{2/\alpha}}{b_n } \esp \pp{\sum_{k=1}^{n}  \sum_{m \le i_1\le m_n<i_2} \frac{1}{\Gamma_{i_1}^{1/\alpha}\Gamma_{i_2}^{1/\alpha}}  \inddd{k\in R_{n,i_1,i_2}}}  \nonumber\\
&=\frac{w_{n}^{2/\alpha}}{b_n }  \sum_{m \le i_1\le m_n<i_2}  \esp\pp{\Gamma_{i_1}^{-1/\alpha} \Gamma_{i_2}^{-1/\alpha}}   \pp{\sum_{k=1}^{n} \proba\pp{k\in R_{n,i_1,i_2}}}
\nonumber \\
&\le C \frac{nw_{n}^{2/\alpha-2}}{b_n}  \sum_{m \le i_1\le m_n<i_2}   i_1^{-1/\alpha}  i_2^{-1/\alpha}   
\le  C \frac  n{b_n}w_{n}^{2/\alpha-2} m_n^{1-1/\alpha}  \rightarrow 0,\label{eq:A_n,2}
\end{align}
where we have used the relation $m_n\gg w_{n}^2/(n \log^{1/( 1-\alpha)}n)$.
For $A_{n,3}$, similarly we have
\equh\label{eq:A_n,3}
\esp A_{n,3} \le C \frac n{b_n} w_{n}^{2/\alpha-2}  \sum_{m_n<  i_1<i_2}   i_1^{-1/\alpha}  i_2^{-1/\alpha} \le C  \frac n {b_n} w_{n}^{2/\alpha-2} m_n^{2-2/\alpha}  \rightarrow 0,
\eque
since the last bound above is of smaller order than that for $\esp A_{n,2}$.

For $A_{n,1}$, since the index $i_1$ takes only finitely many values, and  
  by the monotonicity of  $ \Gamma_i $ in $i$, it is enough to show 
\equh\label{eq:A_n,1}
\P\left( \Gamma_1^{-1/\alpha} Z_n>\epsilon \right)=\P\left( \Gamma_1<Z_n^\alpha \epsilon^{-\alpha} \right)\le
\frac{\esp(Z_n^\alpha)}{\epsilon^\alpha}\le C(\esp Z_n)^\alpha \rightarrow 0
\eque
as $n\rightarrow\infty$,
where
\[
Z_n=\frac{w_{n}^{2/\alpha}}{b_n}\max_{k=1,\ldots,n}  \sum_{ i>m_n} \frac{ 1}{(\Gamma_{i}-\Gamma_1)^{1/\alpha}}  \inddd{k\in R_{n,1,i}}  
\]
is independent of $\Gamma_1$.
Bounding max of non-negative terms by sum, we have
\[
\esp Z_n  \le \frac{w_{n}^{2/\alpha}}{b_n}\sum_{k=1}^{n} \esp\left(\sum_{ i>m_n} \frac{ 1}{(\Gamma_{i}-\Gamma_1)^{1/\alpha}}  \inddd{k\in R_{n,1,i}}\right)\le C \frac n{b_n} w_{n}^{2/\alpha-2} m_n^{1-1/\alpha}.
\]
The last bound above is of the same order as the bound above for $\esp  A_{n,2} $. This completes the proof.
\end{proof}
\section{Convergence of clusters}\label{sec:single}

\subsection{Main result and overview of the proof}\label{sec:overview 2}
We have introduced the space $S = \wt\ell_0\setminus\{[\vv0]\}$. To address the convergence of a single cluster, we use the formulation of convergence of \emph{clusters} 
in \cite[Section 6.2]{kulik20heavy}. In particular, we recall the boundedness $\wt\calB_0$ of $S$, the space of $\wt\calB_0$-boundedly finites on $S$, and the vague$^\#$ convergence in Appendix \ref{sec:topo}. 
Recall 
\[
\vv X_{d_n}:= \vv X_{d_n,1} = (X_1,\dots,X_{d_n}),
\]
and we also view it as an element in $\wt\ell_0$. We introduce   on $S$  the following   measures belonging to    $\cal{M}_{\wt{\cal{B}}_0}$:
\[
{\vv\nu}_{d_n}\equiv {\vv\nu}_{d_n,b_n}  :=  { k_n \esp\pp{\delta_{b_n^{-1} \vv{X}_{d_n}};\vv X_{d_n}\ne\vv0}} = k_n \esp\pp{\delta_{b_n^{-1} \vv{X}_{d_n}}},\quad n\in \N,
\]
  where the last equality above holds since $X_1>0$ with probability $1$ {in view of \eqref{eq:p>=1} (note, however, this is not the case  when we work with the approximations below)}.
For the sake of simplicity most of the time we do not mark explicitly the dependence on $b_n$. 
  Recall $b_n = ((1/2)
n\log n)^{1/\alpha}$ and $k_n = \floor{n/d_n}$. 
  Define the  following cluster measure 
\[
{{\vv\nu}}^{(\rho)}:  = \vartheta_\rho \int_0^\infty \esp\left[ \delta_{r \vv Q}\right] \alpha
 r^{-\alpha-1} dr,\quad \rho\in [0,1],
\]
 {where $\vv Q$ is the conditional spectral tail process introduced in (\ref{eq:Q = Theta}).}
So ${\vv\nu}_{d_n},{\vv\nu}\topp\rho\in\cal{M}_{\wt{\cal{B}}_0}$. Recall $\vartheta_\rho = (1-2\rho\beta)\mathsf q_{F,2}$, and $\theta = \vartheta_1$.  

Our second main result is the following on the convergence of cluster measures.
\begin{Thm}\label{thm:2}
 Suppose 
 \equh\label{eq:d_n assump}
{\limn \frac{\log d_n}{\log n} = \rho\in(0,1] \qmand d_n = o(n) \mbox{ when } \rho = 1. }
 \eque Then,  
\begin{align*}
{\vv\nu}_{d_n}    \overset{v^\#}{\longrightarrow} {\vv\nu}^{(\rho)} 
\end{align*}
in $\calM_{\wt\calB_0}$ as $n\to\infty$. 
\end{Thm}
\begin{proof}[Proof of Theorem \ref{thm:2'}]
 As an immediate corollary of Theorem \ref{thm:2}, we obtain the statement of Theorem \ref{thm:2'} under the assumption \eqref{eq:d_n assump} for $\rho\in (0,1]$. It remains to prove \eqref{eq:block EVT} for $\rho=0$.  The argument relies on the following general monotonicity fact.
Recall that $k_n = \floor{n/d_n}$. Suppose $d_n\ll d_n'$,  and define $k_n' = \floor{n/d_n'}$. Then
\equh\label{eq:block-monotonicity}
\liminf_{n\to\infty} k_n\P\left(\max_{i=1,\ldots,d_n}X_i>b_n\right)\geq 
\limsup_{n\to\infty} k_n'\P\left(\max_{i=1,\ldots,d_n'}X_i>b_n\right).
\eque
Indeed, to see \eqref{eq:block-monotonicity} applying an union bound gives
\begin{align}\label{eq:union-bound}
k_n'\P\left(\max_{i=1,\ldots,d_n'}X_i>b_n\right)&\leq k_n'(d_n'/d_n+1)\P\left(\max_{i=1,\ldots,d_n}X_i>b_n\right)\nonumber
\\&=k_n \P\left(\max_{i=1,\ldots,d_n}X_i>b_n\right)(1+o(1)).
\end{align}
Another union bound and the definition of $b_n$ yield 
\begin{align*}
\limsup_{n\to\infty} k_n\P\left(\max_{i=1,\ldots,d_n}X_i>b_n\right)\leq \limsup_{n\to\infty}k_n d_n\P(X_1>b_n)=1.
\end{align*}
Now, the conclusion \eqref{eq:block-monotonicity} follows by taking the limits in \eqref{eq:union-bound}.

To  show \eqref{eq:block EVT} for $\rho=0$, we   use \eqref{eq:block-monotonicity} with $d_n'$ such that $\limn\log d_n'/\log n = \rho$ for an arbitrary $\rho\in (0,1)$ and $d_n\to\infty, d_n = o(\log n)$. We have then 
\begin{align*}
\liminf_{n\to\infty} k_n\P\left(\max_{i=1,\ldots,d_n}X_i>b_n\right)\geq 
(1-2\rho \beta)\mathsf{q}_{F,2},
\end{align*}
and since $\rho$ is arbitrary, we get 
\begin{align*}
\liminf_{n\to\infty} k_n\P\left(\max_{i=1,\ldots,d_n}X_i>b_n\right)\geq 
\mathsf{q}_{F,2}.
\end{align*}
At the same time, if $\tilde d_n=o(\log n)$ then the anticlustering condition $\AC(d_n,b_n)$ is fulfilled (see the comments after Theorem \ref{thm:2'}), and hence {by \cite[Corollary 6.2.6]{kulik20heavy}}, we have
\begin{align*}
\lim_{n\to\infty} \tilde k_n \P\left(\max_{i=1,\ldots,\wt{d}_n}X_i>b_n\right)=\mathsf{q}_{F,2},
\end{align*}
where $\tilde k_n=\sfloor{n/\tilde d_n}$;  Using again \eqref{eq:block-monotonicity}, this time with $d_n=\tilde d_n$ and $d_n'$ slowly varying, we can apply a squeeze argument to conclude \eqref{eq:block EVT} for $\rho=0$. 
\end{proof}
\color{magenta}
\color{black}
 
{\bf The proof of Theorem \ref{thm:2} proceeds again by a series of approximations.}
However, there are three key differences compared to the proof of Theorem \ref{thm:1}. 
\begin{enumerate}[(i)]
\item The proof this time consists of three steps, the de-aggregation, the mean approximation, and the truncation, with the previous step of Poisson approximation replaced by the mean approximation. The de-aggregation step follows immediately from the previous de-aggregation step in Proposition \ref{prop:5.2}. 
\item The mean approximation is part of the previous Poisson approximation and straightforward. The absence of a Poisson approximation means that there is  no need to control the second moment, and hence the restriction $m_n\ll w_n^2\log n/n$ {(so $\limsupn\log m_n/\log n\le 1-2\beta$)}
 is not needed. Therefore, the upper constraint on $m_n$ now comes from the de-aggregation step, which imposes 
\begin{equation}\label{eq:m_n clust upp}
m_n\ll \frac{w_n}{d_n^\beta}{\log n}
 \qmwith \rho\in(0,1].
\end{equation}
{(So $\limsupn\log m_n/\log n\le 1-(1+\rho)\beta$.)}
\item At the same time, the estimate of the truncation becomes much more involved than before, and this time we need to impose 
\begin{equation}\label{eq:m_n clust low}
m_n\gg \frac{w_n}{d_n^\beta}\frac 1{\log n},
\end{equation}
 which replaces  the requirement $m_n\gg w_n^2/(n\log ^{1/(1-\alpha)}n)$ 
  from \eqref{eq:m_n ppc}. Most of the effort is devoted to this step; see Section \ref{sec:truncation single}. 
\end{enumerate}
\begin{Rem}\label{rem:why}
Comparing Theorem \ref{thm:2} with Theorem \ref{thm:1}, we point out an unusual feature: Despite the weak convergence in \eqref{eq:PPC}, the mean measure of the point process  $\xi_n$  does not converge to that of the Poisson point process $\xi$ under a certain rate restriction of the block size $d_n$ as $n\rightarrow\infty$. In particular,  Theorem \ref{thm:2} implies that when {$\limn\log d_n/\log n = \rho$ with $\rho\in(0,1)$}
the mean measure $\esp\xi_n$ converges vaguely to ${\vv\nu}^{(\rho)}\otimes\Leb$, which is strictly greater than the mean measure $\esp\xi={\vv\nu}^{(1)}\otimes\Leb$ since $\vartheta_\rho>\vartheta_1$. Note that this discrepancy does not lead to a contradiction since, in general, a weak convergence of the point process does not necessitate a convergence of the mean measure. Only the  Fatou-type relation $\liminf_n\esp\xi_n\ge \esp \xi$  follows from the weak convergence, which is consistent with our result.  

Our proofs reveal more of this delicate phenomenon. Note that in \eqref{eq:m_n clust low} $m_n$ is at a faster rate than $n^{1-2\beta}$ in \eqref{eq:m_n ppc}. Our proofs of Theorems \ref{thm:1} and \ref{thm:2} then indicate that including those renewals indexed by $i\in\{n^{1-(2-\delta)\beta},\dots,n^{1-(1+\rho)\beta}\}$ for any $1\le 1+\rho<2-\delta<2$  will not affect the limit behavior at the macroscopic level (of probability of order $O(1)$), but will affect the limit behavior at the mesoscopic level (rare events with probability of order $O(k_n\inv)$). In other words, 
while for the Poisson limit theorem the contributing renewals are those with index $i = O(n^{1-2\beta})$, if one is interested in the asymptotic exceedance probability over a block (of size $d_n = o(n)$) alone, more layers of renewals should be taken into account in the approximation. Necessarily, the exceedance events over different blocks caused from these renewals are asymptotically dependent. A precise characterization of the joint exceedance probabilities over different blocks is of its own interest. We do not pursue this question here.
\end{Rem}
This time we write 
\begin{align*}
\vvX_{d_n,m_n}^* &= ({X}_{n,m_n,1}^*,\dots, {X}_{n,m_n,d_n}^*) \mwith {X}_{n,m_n,k}^*:= 
\pp{\frac{w_{n}}{m_n}}^{2/\alpha}
 \sum_{1\le i_1<i_2\le m_n}  V^*_{i_1,i_2}   \inddd{k\in R_{n,i_1,i_2}},\\
\widetilde{\vvX}_{d_n,m_n}& = (\widetilde X_{n,m_n,1},\dots,\widetilde X_{n,m_n,d_n}) \mwith \widetilde X_{n,m_n,k}:= 
\pp{\frac{w_{n}}{\Gamma_{m_n+1}}}^{2/\alpha}
 \sum_{1\le i_1<i_2\le m_n}  V^*_{i_1,i_2}
    \inddd{k\in R_{n,i_1,i_2}}.  
\end{align*}
The random variables ${X}_{n,m_n,k}^*$ 
 and $\widetilde{X}_{n,m_n,k}$ are the same as $X_{n,k}^*$ and $\widetilde X_{n,k}$ in \eqref{eq:X star whole} 
 and \eqref{eq:X tilde} respectively, but now we indicate the dependence on $m_n$ in the notation. 
We define 
\begin{align}\label{eq:nu*}
{\vv\nu}^*_{d_n,m_n}\equiv {\vv\nu}_{d_n,b_n,m_n}^* &:=k_n\esp(\delta_{\vvX_{d_n,m_n}^*/b_n};\vvX_{d_n,m_n}^*\ne\vv0),\\
\widetilde{{\vv\nu}}_{d_n,m_n}\equiv \widetilde{{\vv\nu}}_{d_n,b_n,m_n} &:=k_n\esp(\delta_{\widetilde{\vvX}_{d_n,m_n}/b_n};\widetilde{\vvX}_{d_n,m_n}\ne\vv0),\label{eq:nu_tilde}\\
{\vv\nu}^\circ_{d_n,m_n}\equiv {\vv\nu}^\circ_{d_n,b_n, m_n}&:=k_n\esp\pp{\sum_{1\le i_1<i_2\le m_n}\delta_{V^*_{i_1,i_2}\vv Q_{n,d_n,1}\topp{i_1,i_2}/c_n}\inddd{R_{n,i_1,i_2}\cap \calI_{d_n,1}\ne\emptyset}}.\nonumber
\end{align}
Introduce 
\begin{equation}\label{eq:eta rho}
\eta\topp\rho:=\sif\ell1 \ddelta{ \vartheta_\rho^{1/\alpha} \Gamma_\ell^{-1/\alpha},\  \vv{Q}_\ell ,\ U_\ell     }, \ \rho\in(0,1].
\end{equation}
So $\eta\topp1 = \eta$.

We first prove the de-aggregation and mean approximation steps in the following single proposition. More precisely, the de-aggregation is \eqref{eq:single 1} below, and the mean approximation is \eqref{eq:single 2} below. 

\begin{Prop}\label{prop:de-agg}
Under the assumptions \eqref{eq:d_n assump}, \eqref{eq:m_n clust upp} 
and 
{\equh\label{eq:m_n rate}
\limn\frac{\log m_n}{\log n} = 1-(1+\wt\rho)\beta \qmwith \wt\rho\in[\rho,1],
\eque
}
 we have
\[
\limn\abs{{\vv\nu}_{d_n,m_n}^*(H) - {\vv\nu}\topp{\wt\rho}(H)} = 0, 
\]
for all $H\in\Clipb(S)$. In particular,
\[
\proba\pp{\frac{\nn{\vv X^*_{d_n,m_n}}_\infty}{b_n}> y}\sim \vartheta_{\wt\rho} d_n \proba\pp{\frac{X_0}{b_n}> y}, \mfa y>0.
\]
\end{Prop}
{For the proof of Theorem \ref{thm:2}, we only need the case $\wt{\rho}=\rho$ from the Proposition \ref{prop:de-agg} due to the lower rate restriction \eqref{eq:m_n clust low}.}

\begin{proof}[Proof of Proposition \ref{prop:de-agg}]
{
We will  make use of results from the de-aggregation step in Section \ref{sec:de-agg}.
First,   under \eqref{eq:d_n assump},  the requirement \eqref{eq:m_n ppc 2} becomes  
\equh\label{eq:single block rate}
\frac{w_{n}^2}{n}\frac1{\log^{1/(1-\alpha)}n}\ll m_n\ll  \frac{w_n}{d_n^\beta}\log n,
\eque
since if $\limn\log d_n/\log n = \rho$ with $\rho>0$, the term $w_n/n^\delta$ on the right-hand side of \eqref{eq:m_n ppc 2} can be dropped by taking $\delta<\rho/\beta$.  Note that \eqref{eq:single block rate} is implied by our assumptions on $m_n$. We then claim that
 for all $H\in\Clipb(S)$,
\equh\label{eq:single 1}
\limn\abs{{\vv\nu}_{d_n,m_n}^*(H)-{\vv\nu}_{d_n,m_n}^\circ(H)} = 0.  
\eque}
Indeed, write $\xi_{n,d_n,m_n}^* = \xi_n^*$ to explicitly indicate the block size $d_n$ and also the truncation level $m_n$ in the notation. 
Let $H\in\Clipb(S)$ and we introduce $f_H\in\Clipb(S\times[0,1])$ by $f_H(\vvx,s) :=H(\vvx)$.  
We have, in view of \eqref{eq:xi_n*}, that
\[
{\vv\nu}_{d_n,m_n}^*(H) = \esp \xi_{n,d_n,m_n}^*(f_H)+o(1) \qmand {\vv\nu}^\circ_{d_n,m_n}(H) = \esp \xi_n^\circ(f_H)+o(1),
\]
where the term $o(1)$ is due to the fact that $d_n$ may not divide $n$ and in this case there is an incomplete block $\{\floor{n/d_n}d_n+1,\dots,n\}$ that does not effect the asymptotic analysis.
Now, \eqref{eq:single 1} follows immediately from Proposition \ref{prop:5.2} as
\[
\abs{{\vv\nu}_{d_n,m_n}^*(H)-{\vv\nu}_{d_n,m_n}^\circ(H)} \le \esp\abs{\xi_{n,d_n,m_n}^*(f_H)-\xi_{n,d_n,m_n}^\circ(f_H)}+o(1)\to 0.  
\]

It remains to show that under the assumption in the proposition, for all $H\in\Clipb(S)$, 
\equh\label{eq:single 2}
\limn\abs{{\vv\nu}_{d_n,m_n}^\circ(H)-{\vv\nu}\topp{\wt\rho}(H)}=0. 
\eque
Introduce \[
\xi\topp{\rho}:=\sif \ell1\ddelta{\vv Q_\ell\vartheta_\rho^{1/\alpha}\Gamma_\ell^{-1/\alpha},U_\ell}.
\]
With \eqref{eq:m_n rate} 
 we claim that  
\equh\label{eq:mean conv1}
\limn \esp \xi_{n,d_n,m_n}^\circ(f) = \esp \xi\topp{\wt\rho}(f),
\eque
for all $f\in \Clipb(S\times[0,1])$, that is, we have the vague convergence of the mean measures. 
Indeed, 
recall $\eta_n$ from (\ref{eq:eta_n}). In the proof of Proposition \ref{prop:5.3}, the relations in the line of 
(\ref{eq:rho_n-a}) is valid for arbitrary $m_n\rightarrow\infty$. In particular, with \eqref{eq:m_n rate} 
we have
\begin{align*}
&\proba(c_n^{-1}V_{1,2}>\epsilon)\proba\pp{R_{n,1,2}\cap \cal{I}_{d_n,1}\neq \emptyset} 
\sim  2 (1-2\tilde{\rho}\beta)\mathsf q_{F,2} \epsilon^{-\alpha} \frac{1}{m_n^2 k_n} = 2 \vartheta_{\wt{\rho}} \epsilon^{-\alpha} \frac{1}{m_n^2 k_n}.
\end{align*}
Hence, the corresponding limiting point process of $\eta_n$   is $\eta^{(\tilde\rho)}$ from \eqref{eq:eta rho}.
{By a standard approximation argument (e.g.,   \cite[Theorem B.1.17]{kulik20heavy} and \cite[Lemma 4.1]{kallenberg17random})}, it follows from \eqref{eq:mean} that
$\limn \esp \eta_n(f) = \esp \eta\topp{\wt\rho}(f)$ for all $f\in \Clipb((0,\infty]\times S^*\times[0,1])$.
Then based on the transformation $T$ in \eqref{eq:T},  we can derive \eqref{eq:mean conv1}; {see again \cite[Proposition 5.5]{resnick07heavy}}. 
Therefore, as $n\rightarrow\infty$,
\[
\abs{{\vv\nu}_{d_n,m_n}^\circ(H)-{\vv\nu}\topp{\wt\rho}(H)}\le \abs{\esp\xi_{n,d_n,m_n}^\circ (f_H)-\esp \xi\topp{\wt\rho}(f_H)}+o(1)\to 0,
\]
as desired.\end{proof}

\subsection{Truncation}\label{sec:truncation single}
It remains to show the truncation step. We shall establish the following.
\begin{Prop}\label{prop:truncation}
Under \eqref{eq:d_n assump} and  \eqref{eq:m_n clust low},
we have for all $H\in\Clipb(S)$, 
\[
\limn\abs{{\vv\nu}_{d_n,m_n}(H) - {\vv\nu}_{d_n,m_n}^*(H)} = 0.
\]
\end{Prop}

The key to the truncation estimate is the following lemma. Interestingly, its counterpart Lemma \ref{Lem:remainder whole n} is much simpler (see Remark \ref{rem:curious}).
\begin{Lem}\label{lem:5.5'}
Assume \eqref{eq:m_n clust low}.
Then, for all $\epsilon>0$,
\[
 \proba\pp{\frac1{b_n}\max_{k=1,\dots,d_n}\abs{X_{n,k}-\wt X_{n,m_n,k}}>\epsilon} = o\pp{k_n\inv}
\]
as $n\rightarrow\infty$.
\end{Lem}
\begin{proof}
We introduce
\begin{align*}
A_{n,1}\equiv A_{n,d_n,m_n,1}  & := \frac{w_{n}^{2/\alpha}}{b_n}\max_{k=1,\ldots,d_n}\sum_{1\le i_1 \le m,  m_n<i_2}   \frac{1}{\Gamma_{i_1}^{1/\alpha}\Gamma_{i_2}^{1/\alpha}}\inddd{k\in   R_{n,i_1,i_2}},
\\
A_{n,2}\equiv A_{n,d_n,m_n,2}  & := \frac{w_{n}^{2/\alpha}}{b_n}\max_{k=1,\ldots,d_n}\sum_{m< i_1 \le m_n<i_2}   \frac{1}{\Gamma_{i_1}^{1/\alpha}\Gamma_{i_2}^{1/\alpha}}\inddd{k\in   R_{n,i_1,i_2}},
\\
A_{n,3}\equiv A_{n,d_n,m_n,3} &:=  \frac{w_{n}^{2/\alpha}}{b_n} \max_{k=1,\ldots,d_n}\sum_{  m_n<i_1<i_2 }   \frac{1}{\Gamma_{i_1}^{1/\alpha}\Gamma_{i_2}^{1/\alpha}}\inddd{k\in   R_{n,i_1,i_2}}. 
\end{align*}

This time, we shall prove 
\[
\proba(A_{n,d_n,m_n,i}>\epsilon) = o\pp{k_n\inv}, i=1,2,3.
\]
For $i=2,3$, it is the same argument as before and the only change is replacing $\summ k1n$ by $\summ k1{d_n}$ when bounding the maximum by the sum (see \eqref{eq:A_n,2} and \eqref{eq:A_n,3}). Thus we have this time
\begin{align*}
\esp A_{n,d_n,m_n,2} &\le C\frac{d_n}{b_n}w_n^{2/\alpha-2}m_n^{1-1/\alpha},\\
\esp A_{n,d_n,m_n,3} &\le C\frac{d_n}{b_n}w_n^{2/\alpha-2}m_n^{2-2/\alpha}.
\end{align*}
For these two expectations to go to zero at the rate $o(k_n\inv)$, it suffices to impose $m_n\gg w_n^2\log ^{1-\alpha} n/n$. 

Now we deal with $A_{n,1}$. This time we apply a more refined analysis. We start with
\[
\proba\pp{A_{n,d_n,m_n,1}>\epsilon} \le m\proba\pp{\frac{w_n^{2/\alpha}}{b_n}\frac1{\Gamma_1^{1/\alpha}}\max_{k=1,\dots,d_n}\sum_{i_2>m_n}\frac1{\Gamma_{i_2}^{1/\alpha}}\inddd{k\in R_{n,1,i_2}}>\frac\epsilon m}.
\]
We shall show
\equh\label{eq:hard}
\proba\pp{\frac{w_n^{2/\alpha}}{b_n}\frac1{\Gamma_1^{1/\alpha}}\max_{k=1,\dots,d_n}\sum_{i\in I}\frac1{\wt\Gamma_{i}^{1/\alpha}}\inddd{k\in R_{n,1,i}}>\epsilon} = o(k_n\inv), 
\eque
where $\{\wt\Gamma_i\}_{i\in\N}$ is an independent copy of the sequence $\{\Gamma_i\}_{i\in\N}$, and $I = I_{n,r}, r=1,2,3$ respectively defined below: {choose $\delta\in (0,1-\beta)$ and} set $w_{n,\delta} :=\floor{n^{1-\beta-\delta}}$, and
\begin{align*}
I_{n,1}\equiv I_{n,1,\delta}&:=\{m_n,\dots,w_{n,\delta}\}\\
I_{n,2}\equiv I_{n,2,\delta}&:=\{w_{n,\delta}+1,\dots,w_n\},\\
I_{n,3}&:=\{w_n+1,\dots\}.
\end{align*}

We first have
\begin{align}
\proba &\pp{\frac{w_n^{2/\alpha}}{b_n}\frac1{\Gamma_1^{1/\alpha}}\max_{k=1,\dots,d_n}\sum_{i\in I_{n,3}}\frac1{\wt\Gamma_{i}^{1/\alpha}}\inddd{k\in R_{n,1,i}}>\epsilon} \nonumber\\
&\le \frac{d_n}{w_n}\proba\pp{\frac{w_n^{2/\alpha}}{b_n}\frac1{\Gamma_1^{1/\alpha}}\sum_{i>w_n}\frac1{\wt\Gamma_{i}^{1/\alpha}}\inddd{1\in R_{n,i}}>\epsilon}\nonumber\\
& \le \frac{d_n}{w_n}\proba\pp{\Gamma_1\le \frac{w_n^2}{\epsilon^\alpha b_n^\alpha}\pp{\sum_{i>w_n}\frac1{\wt \Gamma_i^{1/\alpha}}\inddd{1\in R_{n,i}}}^\alpha} \le C \frac{d_nw_n}{b_n^\alpha}\esp\pp{\sum_{i>w_n}\frac1{\wt \Gamma_i^{1/\alpha}}\inddd{1\in R_{n,i}}}^\alpha.\label{eq:union}
\end{align}
Then, using $\esp Z^\alpha\le (\esp Z)^\alpha$ (since $\alpha\in(0,1)$), $\esp \Gamma_i^{-1/\alpha}\le C i^{-1/\alpha}$ for $i$ large enough, {and $\proba(1\in R_{n,i})=1/w_n$} we have
\equh\label{eq:I_{n,3}}
\proba \pp{\frac{w_n^{2/\alpha}}{b_n}\frac1{\Gamma_1^{1/\alpha}}\max_{k=1,\dots,d_n}\sum_{i\in I_{n,3}}\frac1{\wt\Gamma_{i}^{1/\alpha}}\inddd{k\in R_{n,1,i}}>\epsilon}  \le C \frac{d_nw_n}{b_n^\alpha} \frac{w_n^{\alpha-1}}{w_n^\alpha} = \frac{C}{k_n\log n}.
\eque
We next have, following the same step \eqref{eq:union} above, 
\begin{align*}
\proba &\pp{\frac{w_n^{2/\alpha}}{b_n}\frac1{\Gamma_1^{1/\alpha}}\max_{k=1,\dots,d_n}\sum_{i\in I_{n,2,\delta}}\frac1{\wt\Gamma_{i}^{1/\alpha}}\inddd{k\in R_{n,1,i}}>\epsilon} \\
& \le C\frac{d_nw_n}{b_n^\alpha}\esp{ \pp{\sum_{i=w_{n,\delta+1}}^{w_n}\frac1{\wt \Gamma_i^{1/\alpha}}\inddd{1\in R_{n,i}}}^\alpha}
\le C\frac{d_nw_n}{b_n^\alpha}\esp{ \pp{\sum_{i=w_{n,\delta+1}}^{w_n}\frac1{\wt \Gamma_i}\inddd{1\in R_{n,i}}}}\\
& \le C \frac{d_nw_n}{b_n^\alpha}\frac{\delta \log n}{w_n} \le \frac{C \delta }{k_n},
\end{align*}
for some constant $C$ independent from $\delta,n$. This time we used $(\sum_i c_i)^\alpha\le \sum_ic_i^\alpha$ for non-negative $c_i$ and $\alpha\in(0,1)$. So, 
\equh\label{eq:I_{n,2}}
\lim_{\delta\downarrow 0}\limsupn
k_n\proba \pp{\frac{w_n^{2/\alpha}}{b_n}\frac1{\Gamma_1^{1/\alpha}}\max_{k=1,\dots,d_n}\sum_{i\in I_{n,2,\delta}}\frac1{\wt\Gamma_{i}^{1/\alpha}}\inddd{k\in R_{n,1,i}}>\epsilon}= 0.
\eque
Note that the assumption on $m_n$ is not needed in estimating \eqref{eq:hard} with $I = I_{n,3}$ and $I = I_{n,2}$.

Now we look at \eqref{eq:hard} with $i$ restricted to $i\in I_{n,1}$.
First, recalling the renewal model in Section \ref{sec:renewal}, we have {(cf. \eqref{eq:proba intersection-a})}
\[
r_n:=\proba(R_{n,1}\cap\calI_{d_n,1}\ne \emptyset) \sim \mathsf C_F \frac{d_n^{1-\beta}}{n^{1-\beta}}\sim \mathsf C_F \frac{1}{k_n^{1-\beta}}.
\]
On the other hand, for all $\delta>0$ there exists $p_\delta\in\N$ such that for
\[
\Omega_{n,2,\delta}:=\ccbb{\max_{k=1,\dots,d_n}\sum_{i=m+1}^{w_{n,\delta}}\inddd{k\in R_{n,1,i}}\le p_\delta},
\]
we have $\proba(\Omega_{n,2,\delta}^c) = o(k_n\inv)$. Indeed, this follows from
\[
\proba(\Omega_{n,2,\delta}^c) \le d_n \binom{w_{n,\delta}}{p_\delta+1} {\frac1{w_n^{p_\delta+2}}} \le {\frac {d_nw_{n,\delta}^{p_{\delta}+1}}{w_n^{p_\delta+2}}}.
\]
Using that $d_n=O(n/k_n)$, $w_{n,\delta}=O(n^{1-\beta-\delta})$ and $w_n=O(n^{1-\beta})$, the bound above becomes
$$
O\left(\frac{1}{k_n}n^{-\delta(p_{\delta}+1)+\beta}\right)
$$
which is $o(k_n^{-1})$ whenever $p_{\delta}\geq (\beta/\delta-1)_+$.
Then, it follows that
\begin{multline*}
\proba\pp{\frac{w_n^{2/\alpha}}{b_n}\frac1{\Gamma_1^{1/\alpha}}\max_{k=1,\dots,d_n}\sum_{i\in I_{n,1,\delta}}\frac1{\wt\Gamma_{i}^{1/\alpha}}\inddd{k\in R_{n,1,i}}>\epsilon} \\
 \le \proba\pp{\ccbb{\frac{w_n^{2/\alpha}}{b_n}\frac1{\Gamma_1^{1/\alpha}}\max_{k=1,\dots,d_n}\sum_{i\in I_{n,1,\delta}}\frac1{\wt\Gamma_{i}^{1/\alpha}}\inddd{k\in R_{n,1,i}}>\epsilon}\cap\Omega_{n,2,\delta}}+ o(k_n\inv).
\end{multline*}
We bound the first probability on the right-hand side. Since $\Gamma_1$ is exponential, $\proba(\Gamma_1<x)=O(x)$ as $x\to 0$. Hence, that probability  
 is further bounded from above by
\begin{align*}
r_n  \proba &\pp{\ccbb{\frac{w_n^{2/\alpha}}{b_n}\frac1{\Gamma_1^{1/\alpha}}\max_{k=1,\dots,d_n}\sum_{i\in I_{n,1,\delta}}\frac1{\wt\Gamma_{i}^{1/\alpha}}\inddd{k\in R_{n,1,i}}>\epsilon}\cap\Omega_{n,2,\delta}\mmid R_{n,1}\cap\calI_{d_n,1}\ne\emptyset}\\
&\le \frac C{k_n^{1-\beta}} \proba\pp{\Gamma_1\le \frac{w_n^{2}}{\epsilon^\alpha b_n^\alpha}\frac{p_\delta^\alpha}{\wt\Gamma_{m_n}}} \\
&\le \frac C{k_n^{1-\beta}} \proba\pp{\Gamma_1\le \frac{w_n^{2}}{\epsilon^\alpha b_n^\alpha}\frac{p_\delta^\alpha}{\wt\Gamma_{m_n}},\frac{m_n}{\wt\Gamma_{m_n}}\leq 2} + 
\frac C{k_n^{1-\beta}} \proba\left({\frac{m_n}{\wt\Gamma_{m_n}}>2}\right)\\
&\le \frac C{k_n^{1-\beta}} \proba\pp{\Gamma_1\le 2\frac{w_n^{2}}{\epsilon^\alpha b_n^\alpha}\frac{p_\delta^\alpha}{{m_n}}} + 
\frac C{k_n^{1-\beta}} \proba\left(\left|\frac{m_n}{\wt\Gamma_{m_n}}-1\right|>1\right)\\
&
\le C\frac1{k_n^{1-\beta}}\frac{w_n^2}{b_n^\alpha m_n}+C\frac1{k_n^{1-\beta}}\frac{1}{m_n^q},
\end{align*}
where $q$ can chosen in such the way that $\frac1{k_n^{1-\beta}}\frac{1}{m_n^q}=o(k_n^{-1})$; see the bound for $\wt\Gamma_{m_n}$ in \eqref{eq:gamma concentration}.
Therefore, the bound is 
\begin{align*}
\frac{ C}{k_n}\frac{1}{m_n\log n}\frac{w_n}{d_n^{\beta}}+o(k_n^{-1})=o(k_n^{-1}),\label{eq:I_{n,1}}
\end{align*}
whenever $m_n$ satisfies \eqref{eq:m_n clust low}.
(The constant $C$ above depends on $\delta$.)
We have thus proved that, with $m_n$ satisfying \eqref{eq:m_n clust low},
\equh\label{eq:I_{n,1}}
\proba \pp{\frac{w_n^{2/\alpha}}{b_n}\frac1{\Gamma_1^{1/\alpha}}\max_{k=1,\dots,d_n}\sum_{i\in I_{n,1,\delta}}\frac1{\wt\Gamma_{i}^{1/\alpha}}\inddd{k\in R_{n,1,i}}>\epsilon}= o\pp{k_n\inv}.
\eque
Combining \eqref{eq:I_{n,3}}, \eqref{eq:I_{n,2}}, and \eqref{eq:I_{n,1}}, we have proved that $\proba(A_{n,d_n,m_n,1}>\epsilon) = o(k_n\inv)$, under the assumption \eqref{eq:m_n clust low}.
\end{proof}
\begin{Rem}
When $\rho>1-\beta$, the analysis of $A_{n,1}$ in the proof above can be slightly easier as follows. We have
\begin{align}
\proba  \pp{\frac{w_n^{2/\alpha}}{b_n}\frac1{\Gamma_1^{1/\alpha}}\max_{k=1,\dots,d_n}\sum_{i>\wt m_n}\frac1{\wt\Gamma_{i}^{1/\alpha}}\inddd{k\in R_{n,1,i}}>\epsilon} 
&\le C \frac{w_n^2}{b_n^\alpha}\esp\pp{\summ k1{d_n}\sum_{i_2>\wt m_n}\pp{\frac1{\wt\Gamma_{i_2}^{1/\alpha}}\inddd{k\in R_{n,1,i_2}}}}^\alpha\nonumber\\
& \le C \frac{w_n^2}{b_n^\alpha}\pp{\summ k1{d_n}\sum_{i_2>\wt m_n}\esp\pp{\frac1{\wt\Gamma_{i_2}^{1/\alpha}}\inddd{k\in R_{n,1,i_2}}}}^\alpha\nonumber\\
& \le C \frac{w_n^2}{b_n^\alpha}\pp{\frac{d_n\wt m_n^{1-1/\alpha}}{w_n^2}}^\alpha\sim  \frac C{k_n} \frac{w_n^{2-2\alpha}\wt m_n^{\alpha-1}d_n^{\alpha-1}}{\log n},\label{eq:I_{n,3}'}
\end{align}
which is $o(k_n\inv)$ if 
\[
\wt m_n \gg \frac{w_n^2}{d_n}\frac1{\log^{1/(1-\alpha)}n}.
\]
In particular, when $\rho>1-\beta$, taking $\wt m_n = w_{n,\delta}\sim n^{1-\beta-\delta}\gg w_n^2/(d_n\log ^{1/(1-\alpha)}n)$ (which can be achieved by taking $\delta\in(0,\rho-(1-\beta))$), the upper bound \eqref{eq:I_{n,3}'} is of order $o(k_n\inv)$, and this upper bound can replace \eqref{eq:I_{n,2}} and \eqref{eq:I_{n,3}}.
\end{Rem}

\begin{Rem}\label{rem:curious}
For $A_{n,1}$, the same argument as around \eqref{eq:A_n,1} would end up short. Indeed, this time with
\[
Z_{n,d_n,m_n}:=\frac{w_{n}^{2/\alpha}}{b_n}\max_{k=1,\ldots,d_n}  \sum_{ i>m_n} \frac{ 1}{(\Gamma_{i}-\Gamma_1)^{1/\alpha}}  \inddd{k\in R_{n,1,i}},
\]
we have that
\begin{align*}
\proba\pp{A_{n,d_n,m_n,1}>\epsilon}&\le m\proba\pp{\Gamma_1^{-1/\alpha}Z_{n,d_n,m_n}>\frac\epsilon m}\le C (\esp Z_{n,d_n,m_n})^\alpha\\
& \le C\pp{\frac{d_n}{b_n}w_n^{2/\alpha-2}m_n^{1-1/\alpha}}^\alpha.
\end{align*}
Therefore, equivalently we need 
\[
m_n\gg \frac{w_n^2}{d_n}\frac1{\log^{1/(1-\alpha)}n}.
\]
But this condition is too restrictive to be useful.
\end{Rem}

\begin{proof}[Proof of Proposition \ref{prop:truncation}]
{Recall the notation from \eqref{eq:nu*} and \eqref{eq:nu_tilde}.} It suffices to prove
\[
\limn|{\vv\nu}_{d_n}(H)-\wt{{\vv\nu}}_{d_n,m_n}(H)| = 0 \qmand
\limn|\wt{{\vv\nu}}_{d_n,m_n}(H)-{\vv\nu}_{d_n,m_n}^*(H)| = 0.
\]

Write
\[
\vvX_{d_n}\equiv (X_{n,1},\dots,X_{n,d_n}). 
\]

We shall need the following estimates:
\begin{align}
\proba\pp{\frac1{b_n}\nn{\vv X_{d_n}-\wt {\vv X}_{d_n,m_n}}_\infty>\epsilon} = o(k_n\inv),\label{eq:1'}\\
\proba\pp{\frac1{b_n}\nn{\vvX^*_{d_n,m_n}-\wt {\vv X}_{d_n,m_n}}_\infty>\epsilon} = o(k_n\inv),\label{eq:2'}\\
\proba\pp{\frac1{b_n}\nn{\vv X_{d_n}}_\infty>\epsilon}\sim \proba\pp{\frac1{b_n}\nn{\wt{\vv X}_{d_n,m_n}}_\infty>\epsilon}\sim \proba\pp{\frac1{b_n}\nn{\vv X^*_{d_n,m_n}}_\infty>\epsilon} = O(k_n\inv),\label{eq:3'}
\end{align}
The estimate \eqref{eq:1'} is established in Lemma \ref{lem:5.5'}, and \eqref{eq:2'} by a similar argument as in Lemma \ref{lem:2'} exploiting the strong law of large numbers for Poisson arrivals, and we omit the details. Then, the probability concerning $\vv X^*_{d_n,m_n}$ in \eqref{eq:3'} is established in Proposition \ref{prop:de-agg}, and the remaining estimates on the two other probabilities in \eqref{eq:3'} follow from the previous estimates.

Now we prove the desired statement. We start with the first part. Consider $H\in\Clipb(S)$, such that  $H(\vv x)=\vv 0$ for all $\| \vv x \|_\infty\le \delta$ for some $\delta>0$.
  Then,    we have   for any $\epsilon>0$,  for $n$ large enough,
\begin{align*}
& |\vv{{\vv\nu}_{d_n}}(H)- \wt{{\vv\nu}}_{d_n,m_n}(H)| \le  k_n \esp    \left| H\pp{\frac{\vv X_{d_n}}{b_n}}- H\pp{\frac{\wt{\vv X}_{d_n,m_n}}{b_n}} \right|   \\
& = k_n \esp    \pp{ \left| H\pp{\frac{\vv X_{d_n}}{b_n}}- H\pp{\frac{\wt{\vv X}_{d_n,m_n} }{b_n}} \right|  ;  \frac{\snn{\vv X_{d_n}   - \wt{\vv X}_{d_n,m_n} }_\infty}{b_n}\le \epsilon  }\\
& \quad + k_n \esp    \pp{ \left| H\pp{\frac{\vv X_{d_n}}{b_n}}- H\pp{\frac{\wt{\vv X}_{d_n,m_n}}{b_n}} \right|  ;  \frac{\|\vv X_{d_n}   - \wt{\vv X}_{d_n,m_n}  \|_\infty}{b_n}>\epsilon  }.
\end{align*}
We bound the first term using the Lipschitz property of $H$, and excluding the event $\max\{\snn{\vv X_{d_n}/b_n}_\infty,\snn{\wt{\vv X}_{d_n}/b_n}_\infty\}\le \delta$ on which the term vanishes. For the second term, we bound it simply by the tail probability using the boundedness of $H$. So we arrive at
\begin{align*}
 |{\vv\nu}_{d_n}(H)- \wt{{\vv\nu}}_{d_n,m_n}(H)| &
 \le       \epsilon C  k_n \pp{ \proba\pp{\frac1{b_n} \|\vv{X}_{d_n,m_n}\|_\infty>\delta} + \proba\pp{\frac1{b_n}\nn{\wt{\vv{X}}_{d_n,m_n}}_\infty>\delta } } \\
 & \quad +  C k_n \proba\left(  \frac1{b_n}\nn{\vv X_{d_n}   - \wt{\vv X}_{d_n,m_n}  }_\infty>\epsilon  \right)   \\
& \le  C \epsilon  +o(1),
\end{align*}
where in the last line the constant $C>0$ does not depend on $\epsilon$ or $n$, and we have used \eqref{eq:1'} and \eqref{eq:3'}.
  The desired result now follows by letting $\epsilon\downarrow 0$.

Next, we show the second part. 
We can then follow a similar decomposition as above, and then this time apply \eqref{eq:2'} and \eqref{eq:3'}. In summary we have
\[
\abs{\wt{\vv\nu}_{d_n,m_n}(H) - {\vv\nu}_{d_n,m_n}^*(H)}
  \le C\epsilon + Ck_n\proba\left( \frac1{b_n} \nn{\vv X_{d_n,m_n}^*   - \wt{\vv X}_{d_n,m_n}   }_\infty>\epsilon  \right)\le C\epsilon + o(1).
\]
This completes the proof.
\end{proof}

With $d_n$ as a constant, we also have a corresponding result.
\color{black}
\begin{Lem}\label{lem:d fixed}
Fix $d\in\N$. Then, 
\[
\proba\pp{\frac1{b_n}\max_{k=1,\dots,d}X_k> x}\sim \pp{\summ j0{d-1}\proba(\Theta_k = 0, k=1,\dots,j)}\proba\pp{\frac{X_0}{b_n}>x},
\] 
as $n\to\infty$.
\end{Lem}
To compare with Theorem \ref{thm:2'}, note that as $d\to\infty$,  
\[
\summ j0{d-1}\proba(\Theta_k = 0, k=1,\dots,j) \sim \mathsf q_{F,2}d \equiv 
\vartheta_0 d,
\]
where when $j=0$, the probability in the summand above is understood as 1.
\begin{proof}
We write 
\begin{align*}
\proba \pp{\max_{k=1,\dots,d}X_{k}>x} 
& 
= \summ j1d \proba\pp{X_{j}>x, \max_{k'=j+1,\dots,d}X_{k'}<x}\\
& = \summ j1d \proba\pp{ \max_{k'=j+1,\dots,d}X_{k'}<x\mmid X_{j}> x}\proba(X_{j}>x).
\end{align*}
The desired result now follows from the convergence of the tail process in \eqref{eq:micro}.
\end{proof}
\appendix
\section{Topological background}\label{sec:topo}
Let $\ell_0$ be the space of real-valued sequences $\vv x=\pp{x_n}_{n\in \Z}$ satisfying
$\lim_{n\rightarrow\pm\infty} x_n=0$. We endow $\ell_0$   with the sup-norm $\|\cdot\|_\infty$  making it a Banach space. We will also work with the punctured space $\ell_0\setminus\{\vv 0\}$, where    $\vv 0$ denotes the sequence of zeros.    The  boundedness ${\cal{B}}_0$ (i.e., the class of bounded sets) of $\ell_0\setminus\{\vv 0\}$ is formed by subsets $A$ of $\ell_0\setminus\{\vv 0\}$ satisfying $\inf_{\vv x\in A} \|\vv x\|_\infty>0$.

Recall $\wt{\ell}_0$ is  the space of equivalent classes of $\ell_0$ (see Section \ref{sec:main}). We use $[\vv{x}] \in \wt{\ell}_0 $ to denote the equivalent class represented by $\vv x\in \ell_0$. Define a metric on $\wt{\ell}_0$ as  $\wt{d}([\vv x],[\vv y])=\inf_{\vv{x}'\in [\vv{x}],\vv{y}'\in [\vv{y}]} \|\vv x'- \vv y'\|_\infty$,  where $[\vv x]$, $ [\vv y]\in \wt{\ell}_0$. It is known that $(\wt{\ell}_0,\wt{d})$ is separable and complete (see \cite{basrak18invariance}).
 
 A Borel measure $\mu$  on $\ell_0$ induces a Borel measure on $\wt{\ell}_0$  via the measurable mapping $\vv{x}\mapsto [ \vv x]$, which we still denote using the same notation $\mu$. Note that the induced $\mu$ on $\wt{\ell}_0$ can be identified as the original $\mu$ on $\ell_0$ when it is restricted to shift-invariant Borel subsets  of $\ell_0$, or equivalently, to shift-invariant  non-negative or integrable Borel measurable functions on $\ell_0$.     Similar convention applies to the subspace $S = \wt \ell_0\setminus\{[\vv0]\}$.

We shall often work with the subspace 
\[
 S = \wt \ell_0\setminus\{[\vv0]\},
\]
 whose boundedness $\wt{\cal{B}}_0$  is formed by subsets $A$ of $S$ satisfying $\inf_{[\vv x]\in A}\wt{d}([\vv x],[\vv 0])>0$.  
Let $\cal{M}_{\wt{\cal{B}}_0}=\cal{M}_{\wt{\cal{B}}_0}(S)$ denote the space of $\wt{\cal{B}}_0$-boundedly finite measures (i.e., each $\mu\in \cal{M}_{\wt{\cal{B}}_0}$ satisfies $\mu(A)<\infty$ for any Borel $A\in \wt{\cal{B}}_0$) on $S$.
We then equip $\cal{M}_{\wt{\cal{B}}_0}$  with the so-called vague$^\#$ topology, the smallest topology that ensures that the mappings $\mu\mapsto \mu(f)$ are continuous for all bounded continuous $f$ on $S$ with $\wt{\cal{B}}_0$-bounded support, making  $\cal{M}_{\wt{\cal{B}}_0}$  a Polish space. The convergence under the vague$^\#$ topology is called the vague$^\#$ convergence, or simply the vague convergence, denoted as \[\mu_n\overset{v^\#}{\longrightarrow} \mu\]
  as $n\rightarrow\infty$ for  $\mu_n,\mu \in \cal{M}_{\wt{\cal{B}}_0}$, $n\in \N$. This convergence is also characterized by $\mu_n(f)\rightarrow \mu(f)$ for all $f\in \Clipb(S)$, where $\Clipb(S)$ stands for all Lipchitz continuous functions (with respect to the metric $\wt{d}$) with bounded support on $S$.
  See \cite{basrak19note}, \cite[Appendix B.1]{kulik20heavy}, \cite[Section 4.1]{kallenberg17random} for more details about vague$^\#$ topology and convergence.  

\section{Some estimates}
Here we follow the notation in Section \ref{sec:renewal}.
\begin{Lem}\label{lem:rho}
For all increasing sequences $\{\ell_n\}_{n\in\N},\{\ell_n'\}_{n\in\N}$ of integers such that $\ell_n,\ell_n'\to\infty, \ell_n/\ell_n'\to\infty$, there exists a constant $C>0$ such that
\begin{align*}
\rho_{\ell_n,\ell_n'}\topp1 & :=\sum_{ 1\le j_1, j_2\le \ceil{\ell_n/\ell_n'}}  \proba(R_{\ell_n,1,2}\cap \cal{I}_{\ell_n',j_1} \neq \emptyset, \ R_{\ell_n,1,3}\cap \cal{I}_{\ell_n',j_2} \neq \emptyset   )\le  \frac{C \ell_n^{1+\beta}}{w_{\ell_n}^{3}},\\
\rho_{\ell_n,\ell_n'}\topp2& :=\sum_{1\le j_1,j_2\le \ceil{\ell_n/\ell_n'}, j_1\neq j_2} \proba(R_{\ell_n,1,2}\cap \cal{I}_{\ell_n',j_1} \neq \emptyset, \ R_{\ell_n,1,2}\cap \cal{I}_{\ell_n',j_2} \neq \emptyset   )\le \frac{C \ell_n(\ell_n')^{2\beta-1} }{w_{\ell_n}^2}.
\end{align*}
\end{Lem}
\begin{proof}
The above follows from \cite[Lemma 4.7]{bai24phase}. 
Therein, the proof is for $\ell_n = n$ only, but can be easily extended for the above. Note that the upper bound in the last case of \citep[(4.28)]{bai24phase} should be $C nd_n^{\beta_p}\log d_n/w_n^p$, where the factor $\log d_n$ can be replaced by 1 if $\beta_p \ne -1$. The proof was all correct until right before (4.31) therein, where $d_n^{-(|\beta_p|\vee1)}$ should be $d_n^{-(|\beta_p|\wedge 1)}$.
\end{proof}
\begin{Lem}\label{Lem:conv Q}
Let $\ell_n,\ell_n'$ be two sequences of integers increasing to infinity and satisfying  {$\ell_n'\le \ell_n$.} 
\begin{enumerate}[(i)]
\item{ For all $j=1,\dots,\floor{\ell_n/\ell_n'}$, $R_{\ell_n,1}\cap \calI_{\ell_n',j}$ has the same law. Furthermore
\equh\label{eq:proba intersection-a}
\proba\pp{R_{\ell_n,1}\cap \calI_{\ell_n',1}\ne\emptyset}\sim \mathsf C_F\frac{(\ell_n')^{1-\beta}}{w_{\ell_n}} \mmas n\to\infty.
\eque
}
\item For all $j=1,\dots,\floor{\ell_n/\ell_n'}$, $R_{\ell_n,1,2}\cap \calI_{\ell_n',j}$ has the same law. More precisely, we have
\equh\label{eq:proba intersection}
\proba\pp{R_{\ell_n,1,2}\cap \calI_{\ell_n',1}\ne\emptyset}\sim \frac{\ell_n'}{w_{\ell_n}^2}\mathsf q_{F,2} \mmas n\to\infty.
\eque
\item  {Recall $\vv Q_{\ell_n,\ell_n',j}^{(1,2)}=\pp{\inddd{k\in  R_{\ell_n,i_1,i_2}\cap \calI_{\ell_n',j}}}_{k\in\Z}$ from \eqref{eq:Q emp}}. For all $j=1,\dots,\floor{\ell_n/\ell_n'}$, $B\in\calB(S)$ (the Borel $\sigma$-algebra of $S$), $\proba\spp{\vv Q_{\ell_n,\ell_n',j}^{(1,2)}\in B \mid R_{\ell_n,1,2}\cap \cal{I}_{\ell_n',j}\neq \emptyset}$ has the same value. Moreover, we have the total-variation   convergence
\equh\label{eq:Q uniform}
\limn\sup_{B\in \cal{B}(S)} \left| \proba\pp{\vv Q_{\ell_n,\ell_n',1}^{(1,2)}\in B \mmid R_{\ell_n,1,2}\cap \cal{I}_{\ell_n',1}\neq \emptyset}-\proba\pp{\vv Q\in B} \right|=0.
\eque
\end{enumerate}
\end{Lem}
\begin{proof}
The shift invariance in both statements follow from Lemma \ref{lem:Rm}. It remains to examine the asymptotics. 
To show \eqref{eq:proba intersection-a},  we have by the last entrance decomposition
\begin{align*}
\proba(\min (R_{\ell_n,1}\cap \calI_{\ell_n',1})\le \ell_n' )&=\frac{1}{w_{\ell_n}} \sum_{k=1}^{\ell_n'}  \wb{F}(\ell_n'-k)=\frac{1}{w_{\ell_n}} \sum_{k=0}^{\ell_n'-1}  \wb{F}(k).
\end{align*}
Now, the asymptotics follows by choosing $m=\ell_n'$ and using $\wb{F}(x)\sim\mathsf C_Fx^{-\beta}$.

To show \eqref{eq:proba intersection},  we have by the last-entrance decomposition: 
\begin{align}\label{eq:min intersect last}
\proba(\min (R_{\ell_n,1,2}\cap \calI_{\ell_n',1})\le \ell_n' )&=\sum_{k=1}^{\ell_n'} \proba\pp{ \max(R_{\ell_n,1,2}\cap \calI_{\ell_n',1})= k }\nonumber\\
&= \frac{1}{w_{\ell_n}^2} \sum_{k=1}^{\ell_n'}  \wb{F}^{*}(\ell_n'-k)=\frac{1}{w_{\ell_n}^2} \sum_{k=0}^{\ell_n'-1}  \wb{F}^{*}(k),
\end{align}
where $\wb{F}^{*}$ is the tail probability for the inter-arrivals of $\vv{\tau}^{(1)}\cap \vv{\tau}^{(2)}$ satisfying $\wb{F}^{*}(k)\rightarrow \mathfrak{q}_{F,2}$ as $k\rightarrow\infty$. 

In the proof we regard $B$ as a shift-invariant Borel set of $\ell_0$ (i.e., if for every $\vv x\in B$, we have $\mathsf B(\vv x)\in B$ where $\mathsf B$ is the shift operator), and regard $\vv Q$ and related processes as random elements in $\ell_0$. 

Next, by first-entrance decomposition, 
\begin{align}
\proba & \pp{\vv Q_{\ell_n,\ell_n',1}^{(1,2)}\in B \mmid R_{\ell_n,1,2}\cap \cal{I}_{\ell_n',1}\neq \emptyset}\notag\\&=\frac{1}{\proba\pp{ R_{\ell_n,1,2}\cap \cal{I}_{\ell_n',1}\neq \emptyset}}\sum_{k=1}^{\ell_n'} \proba\pp{\vv Q_{\ell_n,\ell_n',1}^{(1,2)}\in B \mmid  \min R_{\ell_n,1,2}=k} \proba\pp{ \min R_{\ell_n,1,2}=k}\notag\\
&\sim\left(\sum_{1\le k\le r \ell_n'}     +\sum_{r \ell_n'< k\le   \ell_n'} \right) \frac{w_{\ell_n}^2}{\ell_n' \mathfrak{q}_{F,2}} \proba\pp{\vv Q_{\ell_n,\ell_n',1}^{(1,2)}\in B \mmid  \min R_{\ell_n,1,2}=k} \proba\pp{ \min R_{\ell_n,1,2}=k},  \label{eq:Q cond prob}
\end{align}
where $r\in (0,1)$ and the last asymptotic equivalence is uniform in $B\in\cal{B}(S)$.
We claim that 
\equh\label{eq:Q uniform1}
\limn\sup_{k=1,\dots,r\ell_n'}\sup_{B\in \calB(S)}\abs{\proba\pp{\vv Q_{\ell_n,\ell_n',1}^{(1,2)}\in B \mmid  \min R_{\ell_n,1,2}=k} -\proba\pp{\vv Q \in B  }} = 0, \mfa r\in(0,1). 
\eque Indeed, consider $\vv Q|_{  \ell_n'-k } =\pp{\ind_{\{\ell \in \vv \tau^{(1)}\cap \vv \tau^{(2)}\}}}_{0\le \ell \le \ell_n'-k }$ (recall $\vv\tau\topp 1,\vv\tau\topp2$ are the same i.i.d.~copies of renewal processes starting from 0 defining $\vv Q_{\ell_n,\ell_n',1}^{(1,2)}$, and again $\vv Q_{\ell_n'-k}$ is understood as an element in $S$). So we have
\[
\proba\pp{\vv Q_{\ell_n,\ell_n',1}^{(1,2)}\in B \mmid  \min R_{\ell_n,1,2}=k}=\proba\pp{ \vv Q|_{ \ell_n'-k }   \in B  }.
\]
Note that for $1\le k\le r \ell_n'$, we have
\begin{align*}
\sup_{k\le r\ell_n'}\sup_{B\in\calB(S)}\abs{\proba\pp{ \vv Q    \in B  }-\proba\pp{ \vv Q|_{ \ell_n'-k }   \in B  }}&\le 
\sup_{k\le r\ell_n'}\proba\pp{ \vv Q|_{ \ell_n'-k } \neq \vv Q }\\&\le \proba\pp{ \vv\tau^{(1)}\cap \vv\tau^{(2)}\cap \{\ell\in \N_0: \ell \ge \ell_n'-\floor{r \ell_n'}  \}\neq \emptyset},
\end{align*}
which goes to zero 
as $n\rightarrow\infty$ since $ \vv\tau^{(1)}\cap \vv\tau^{(2)}$ is transient. So \eqref{eq:Q uniform1} holds. With this uniform convergence and the fact that 
 \[
 \proba(\min R_{\ell_n,1,2}\le r \ell_n')\sim \frac{r\mathfrak{q}_{F,2}\ell_n'}{w_{\ell_n}^2},\quad n\rightarrow\infty,
\] 
which follows from \eqref{eq:min intersect last}, we conclude that the first  sum in \eqref{eq:Q cond prob} converges to
\[
  r \proba\pp{\vv Q\in B} 
\]
uniformly for $B\in \cal{B}(S)$.
Furthermore, 
the $\limsup$ of the second  sum in \eqref{eq:Q cond prob}  is similarly bounded by  $1-r$ (bounding the probability involving $\vv Q_{\ell_n,\ell_n',1}^{(1,2)}$ by $1$). The desired \eqref{eq:Q uniform} now follows by letting $r\uparrow 1$.
\end{proof}
\section{Some related results}\label{sec:other}
We first relate the law of one-sided spectral tail process $\vv\Theta$ in \eqref{eq:Theta_k} to the law of the conditional spectral tail process $\vv Q$. Notice first that $\vv \Theta$ in \eqref{eq:Theta_k} extends,  uniquely in law,  to a two-sided form $\pp{\Theta_k}_{k\in \Z}$ by setting $\Theta_k=\ind_{\{-k\in \vv\tau\topp 3 \cap \vv\tau\topp 4  \}}$ for $k<0$ where $\vv\tau\topp{i}, i=1,2,3,4$ are i.i.d.~renewal processes. This follows from the an extension of the limit relation \eqref{eq:micro} to include negative $k$ indices, which is straightforward and hence omitted. We let $\vv\Theta$ denote the one-sided tail process as before and $\vv\Theta^{\leftrightarrow} := (\Theta_k)_{k\in\Z}$ denote the two-sided spectral tail process.
Next, following \citep[Definition 5.4.6]{kulik20heavy}, the law of $\vv Q$ is uniquely determined by the following:
\equh\label{eq:5.4.6}
\esp H(\vv Q) = \frac1{\vartheta_0 }\esp\pp{\frac{H(\vv\Theta^\leftrightarrow)}{\sum_{k\in\Z}\Theta_k}},
\eque
for every bounded or non-negative measurable function $H$ on $S$; here we used the fact that $\vv\Theta^\leftrightarrow$ is a zero-one valued process. Recall also $\vartheta_0 = \mathsf q_{F,2} = \proba(\Theta_k = 0, \mfa k\in\N)$. 
We have the following.
\begin{Lem}\label{lem:Q = Theta}
The law of $\vv Q$ (viewed as a random element in $S$ and determined by $\vv\Theta^\leftrightarrow$ via \eqref{eq:5.4.6}), is the same as the law of $\vv\Theta$.  
\end{Lem}
\begin{proof}
 Let
$N_{\vv\Theta^\leftrightarrow} := \sum_{k\in\Z}\Theta_k$ denote the total number of renewals of $\vv\Theta^\leftrightarrow$ and similarly $N_{\vv\Theta} := \sif k0\Theta_k$. 
 When both $\vv \Theta^\leftrightarrow$ and $\vv \Theta$ are viewed as   random elements in $S$, with $H$   as in \eqref{eq:5.4.6},  it suffices to show for each $m\in \N$ that
 \begin{equation}\label{eq:condspec goal}
 \frac{1}{\vartheta_0}\esp\pp{\frac{H(\vv\Theta^\leftrightarrow)}{m}\mmid  N_{\vv\Theta^\leftrightarrow}= m } \proba\pp{N_{\vv \Theta} = m} = \esp\pp{H(\vv\Theta) \mid  N_{\vv\Theta}= m }\proba\pp{N_{\vv\Theta} = m}. 
 \end{equation}
 Note that in $S$, the law of $\vv\Theta^\leftrightarrow$ conditioning on $\{N_{\vv Q^\leftrightarrow}=m\}$ 
   is equal to law of $\vv\Theta$ conditioning on $\{N_{\vv Q}=m\}$. This is because modulo a shift,   both  of them consist of   $m$ consecutive i.i.d.\ inter-renewals with a law  $\calL( T \mid T<\infty )$, where $T=\inf \spp{\vv\tau^{(1)}\cap \vv\tau^{(2)}\cap \N_+}$ is the inter-renewal time (recall $\proba\pp{T=\infty}= \frak{q}_{F,2}=\vartheta_0$).  Therefore, to show \eqref{eq:condspec goal}, it suffices to show 
 $ \proba\pp{N_{\vv\Theta^\leftrightarrow} = m}= \vartheta_0 m   \proba\pp{N_{\vv\Theta} = m}$. To do so, first note that  $N_{\vv\Theta^\leftrightarrow}\EqD N_{\vv \Theta}+N_{\vv \Theta}'-1$, where $N_{\vv \Theta}'$ is an independent copy of $N_{\vv \Theta}$, each following a geometric distribution as described below \eqref{eq:qFp}. Then,
  \begin{align*}
 \proba\pp{N_{\vv\Theta^\leftrightarrow} = m} &=\sum_{j = 0}^{m-1} \proba\pp{N_{\vv\Theta} = j+1,N_{\vv\Theta}' = m-j} = \summ j0{m-1}(1-\mathsf q_{F,2})^{m-1}\mathsf q_{F,2}^2\\
 & =   m (1-\mathsf q_{F,2})^{m-1}\mathsf q_{F,2}^2 =  \vartheta_0m\proba\pp{N_{\vv\Theta} = m},
 \end{align*}
 which completes the proof.
\end{proof}
\color{black}

A consequence of Theorem \ref{thm:2'} concerns convergence to the tail process in $\ell_0$.
Recall that the two-sided tail process $\vv Y=\pp{Y_k}_{k\in \Z}$ associated with $(X_k)_{k\in\Z}$ is given by  $Y_k=L_\alpha \Theta_k$,  and $L_\alpha$ is a standard $\alpha$-Pareto random variable, namely, $\P\pp{L_\alpha>x}=x^{-\alpha}$, $x>1$, which is independent of everything else.  
The tail process $\vv Y$ is related to $(X_k)_{k\in\Z}$  through
\[
\calL\pp{\frac{(X_{-d},\dots,X_d)}{b_n}\mmid X_0>b_n} \to \calL\pp{(Y_{-d},\dots,Y_d)}, \mfa d\in\N.
\]
A priori, a tail process $\vv Y$ of a general regularly varying process may not be in $\ell_0$, but it is obviously so in our example. Then, by a triangular-array argument \citep[Lemma 6.1.1]{kulik20heavy}, it follows that there exists {\em certain} block-length sequence $\{d_n\}_{n\in\N}$ tending to $\infty$, such that for any $y>0$,
\equh\label{eq:TP dn}
\calL\pp{\frac{(\dots,0,X_{-d_n},\dots,X_{d_n},0,\dots)}{b_n y}\mmid X_0>b_n y} \to \calL\pp{\vv Y} 
\eque
as $n\rightarrow\infty$.
The triangular-array argument does not provide any information on the rate of $d_n\to\infty$ in \eqref{eq:TP dn}, and hence \eqref{eq:TP dn} might not hold for a specific growth rate a priori. Theorem  \ref{thm:2'} now implies an upper bound for the growth rate in our example.
This is closely related to the fact that the rate in \eqref{eq:TP dn} must satisfy the anticlustering condition. It is natural to expect that $d_n$ may not grow too fast. On the other hand, once at certain rate the above convergence holds, with a slower rate the convergence remains to hold. 
\begin{Coro}
For \eqref{eq:TP dn} to holds, the block-length sequence $\{d_n\}_{n\in\N}$ must grow slower than any polynomial rate.
\end{Coro}
\begin{proof}
We prove by contradiction. Suppose \eqref{eq:TP dn} holds. Then, it follows that for any $x,y>0$,
\[
\limn\proba\left(  \max_{m\le |j|\le d_n} X_j > b_n x  \mmid  X_0 > b_n  y  \right) = \proba \left(  \sup_{|j|\ge m} Y_j>\frac xy\right).
\]
But the right-hand side tends to $0$ as $m\rightarrow\infty$ since almost surely $Y_j=0$ for large enough $j$. This shows that  $\mathcal{AC}(d_n,b_n)$ holds which contradicts with Corollary \ref{Cor:ac fails}. 
\end{proof}
 Below write $M_{d_n}=\max_{k=1,\dots,d_n}X_k$ {and $\vv x_{1:d}$ for $(x_1,\ldots,x_d)$.}

\begin{Coro} 
Following the  setup of Theorem \ref{thm:2}, 
as $n\rightarrow\infty$, we have the following weak convergences in the space $S$: as $n\rightarrow\infty$,
\[
\cal{L}\left(\frac{\vv X_{1:d_n}}{ b_n }   \mmid M_{d_n}>  b_n \right)  \Rightarrow 
  \cal{L}( L_\alpha \vv Q   ),
\]
and
\[
\cal{L} \left(\frac{\vv  X_{1:d_n}}{M_{d_n}}  \mmid M_{d_n} >b_n \right)  \Rightarrow 
\cal{L}( \vv Q   ).
\]
\end{Coro}
\begin{proof}
Consider $H(\vv x) = 1_{\{\|\vv x\|_\infty>u, \ \vv x/ \|\vv x\|_\infty\in B \}}  $  for $u>1$ and  shift-invariant ${\vv\nu}$-continuity Borel set $B$ in $\ell_0\setminus\{\vv 0\}$,
\[
\frac{\esp\left( H(\vv X_{1:d_n}/  b_n)    , M_{d_n}>  b_n \right)}{\proba\pp{M_{d_n}>  b_n }}\sim \frac{k_n\esp\left( H(\vv X_{1:d_n}/  b_n) \right)}{\vartheta_\rho} \rightarrow u^{-\alpha} \proba(\vv {Q}\in B),
\]
where we have applied \eqref{eq:block EVT} in the first step and Theorem \ref{thm:2} in the second step. This establishes the first claim of this corollary. 
The second claim follows from the first one through a continuous mapping argument: see the arguments for (2.10) and (2.11) in the proof of \cite[Theorem 2.2]{basrak18invariance}. 
\end{proof}
Suppose  $\{b_n\}_{n\in\N}$ is as in \eqref{eq:b_n}.
  We recall
Leadbetter's celebrated \emph{distributional mixing $D$ condition}     \cite{leadbetter74extreme} as follows.  For $n,\ell\in \N$,  define a mixing coefficient as
\[
\alpha_{n,\ell}=\max_{I_1,I_2}\left| \proba\pp{\max_{k\in I_1\cup I_2}X_k \le b_n}-\proba\pp{\max_{k\in I_1 }X_k \le b_n}\proba\pp{\max_{k\in  I_2}X_k \le b_n}  \right|,
\]
where the first maximum is taken over  subsets $I_1,I_2\subset \{1,\ldots,n\}$ such that $\min I_2-\max I_1=\ell$. Note that $\alpha_{n,\ell}$ is non-increasing in $\ell$.
 The condition $D(b_n,\ell_n)$ (often  only  written as $D(b_n)$)  is said to hold    for some integer  sequence $(\ell_n)_{n\in\N}$ satisfying $\ell_n=o(n)$, if 
$
\alpha_{n,\ell_n} \rightarrow 0$ as $n\rightarrow\infty$. For our model $(X_k)_{k\in\Z}$,  the following result concerns a lower rate restriction of $\alpha_{n,\ell_n}$  when $\ell_n$ has an upper rate restriction.

\begin{Coro}
Fix  an arbitrary   $\rho\in (0,1)$. We have
\[
\limsupn n^{1-\rho}\alpha_{n,\ell_n}>0
\]
for any sequence $(\ell_n)$ satisfying $\ell_n=o(n^{\rho})$. 
\end{Coro} 
\begin{proof}
We prove by contradiction.  Suppose there exists $\rho\in (0,1)$,  such that with the block size   $d_n\sim n^{\rho}$ and the number of blocks   $k_n\sim n^{1-\rho}$,  we have for some $\ell_n=o(n^{\rho})$, that the relation $ \lim_n k_n\alpha_{n,\ell_n}=0$ holds instead.   The family of  sequences $u_n(\tau)$, $\tau>0$, that satisfies the relation (2.6) of \cite{{leadbetter83extremesPTRF}} can be chosen to be $b_n \tau^{-1/\alpha}$ in view of \eqref{eq:X_0 tail}. Applying \cite[Theorem 3.4]{leadbetter83extremesPTRF} and  Theorem \ref{thm:2'} above (the latter verifies relation (3.8) of \cite{{leadbetter83extremesPTRF}}),  we conclude that the extremal index of $(X_k)_{k\in\Z}$ is $\vartheta_\rho$. This, however, contradicts with the fact that $\vartheta_\rho<\theta$ for $\rho\in (0,1)$ and Corollary  \ref{coro:1'}.
\end{proof}

\bibliographystyle{apalike}
\bibliography{references,references18}
\end{document}